\documentclass{amsart}
\usepackage{amsfonts}
\usepackage{lineno}
\usepackage{amsmath}
\usepackage{amssymb}
\usepackage{amsthm}
\usepackage{cancel}
\usepackage{tikz-cd}
\usepackage{arydshln}
\usepackage{graphicx}
\usepackage{lipsum}
\usepackage{graphics}
\usepackage{tikz}
\usepackage{multicol}
\usepackage{subfigure}
\usepackage{xcolor}
\usepackage{multicol}
\usepackage{microtype}
\usepackage{vwcol}
\usepackage{changepage}
\usepackage{caption}
\usepackage{msc}

\newcommand{\R}{\mathbb{R}}
\newcommand{\C}{\mathbb{C}}

\newcommand{\D}{\mathbb{D}}

\newtheorem{theorem}{Theorem}[section]

\newtheorem{lemma}{Lemma}[section]
\newtheorem{proposition}{Proposition}[section]

\begin{document}
\title{Iterative Variable-Blaschke factorization}
\author{MAXIME LUKIANCHIKOV}
\author{VLADYSLAV NAZARCHUK}
\author{CHRISTOPHER XUE}

\address[Maxime Lukianchikov]{Department of Mathematics, Yale University, 10 Hillhouse Avenue, New Haven, CT 06511}
\email{maxime.lukianchikov@yale.edu}
\address[Vladyslav Nazarchuk]{Department of Mathematics, Yale University, 10 Hillhouse Avenue, New Haven, CT 06511}
\email{vladyslav.nazarchuk@yale.edu}
\address[Christopher Xue]{Department of Mathematics, Yale University, 10 Hillhouse Avenue, New Haven, CT 06511}
\email{christopher.xue@yale.edu}


\begin{abstract}
    Blaschke factorization allows us to write any holomorphic function $F$ as a formal series
        \[ F = a_0 B_0 + a_1 B_0 B_1 + a_2 B_0 B_1 B_2 + \cdots \]
    \noindent{where $a_i \in \C$ and $B_i$ is a Blaschke product. We introduce a more general variation of the canonical Blaschke product and study the resulting  formal series. We prove that the series converges exponentially in the Dirichlet space given a suitable choice of parameters if $F$ is a polynomial and we provide explicit conditions under which this convergence can occur. Finally, we derive analogous properties of Blaschke factorization using our new variable framework.}
\end{abstract}

\maketitle

\section{Introduction} \label{sec: Introduction}

    \subsection{Fourier Series.}

    \noindent{Any analytic function $F$ can be expressed as a power series}  
        \[F(z)=\sum_{k=1}^{\infty}a_k z^k \textrm{\hspace{10mm}} a_k\in \mathbb{C}.\]
    \noindent{Restricting this power series to the boundary of the unit disk in the complex plane results in a classical Fourier series. A dynamical way of constructing this power series, analogous to that which will be employed later, is as follows. Firstly, write $F(z)$ as}
        \[F(z)=F(0)+(F(z)-F(0)).\]
    \noindent{Notice that this guarantees that the function $F(z)-F(0)$ has at least one root at 0. Factoring this root yields}
        \[F(z)=F(0)+z \cdot G_1(z) \]
    \noindent{where $G_1(z)$ is holomorphic. Repeating the same process with $G_1(z)$ yields}
        \begin{align*}
           F(z)&=F(0)+z(G_1(0)+(G_1(z)-G_1(0)) \\ &=F(0)+z(G_1(0)+z(G_2(z)))
        \end{align*}
    \noindent{where $G_2(z)$ is once again analytic. Indeed, one could iterate this procedure ad infinitum (unless $F$ is a polynomial, in which case only a finite number of iterations would be needed), resulting in the formal series}
        \begin{align*}
            F(z) &= F(0)+z(G_1(0)+z(G_2(0)+z(G_3(0)+\cdots \\
            &= F(0)+G_1(0)z+G_2(0)z^2+G_3(0)z^3+\cdots.
        \end{align*}
    
    \subsection{Blaschke Factorization}
    
        \noindent{Iterative Blaschke factorization is the natural nonlinear analogue of Fourier series. Let $F$ be a polynomial in the complex plane. Then, any $F$ can be decomposed as $F(z) - F(0) = B \cdot G$ where}
            \[ B(z) = \prod_{i \in I} \frac{z - \alpha_i}{1 - \overline{\alpha_i} z} \]
        \noindent{is a \textit{Blaschke product} with zeros $\alpha_1,\alpha_2,\ldots$ inside the unit disk $\D$ and $G$ has no roots in $\D$; that is,}
            \[ G(z) = \prod_{|\alpha_i| > 1} \left( z - \alpha_i \right) \prod_{0 < |\alpha_i| < 1} \left( 1 - \overline{\alpha_i}z \right). \]
        
        \noindent{Blaschke factorization is well-defined for any bounded analytic function (see \cite{Garnett}). However, we restrict our study to polynomials. Since holomorphic functions can be approximated by polynomials up to arbitrarily small error on any compact domain and since Coifman, Steinerberger, and Wu showed in \cite{CoifmanStefanWu} that functions of intrinsic-mode type (a classical model for signals) behave essentially like holomorphic functions, we expect that our results have more general extensions in both pure and applied areas. \\}
        
        \noindent{This Blaschke factorization can be repeated iteratively in a similar dynamical framework as the Fourier series. Since the function $G_i(z) - G_i(0)$ has at least one root within the unit disk, it will yield the nontrivial Blaschke factorization $G_i(z) - G_i(0) = B_{i+1} \cdot G_{i+1}$. Therefore, we write}
            \begin{align*}
                F - F(0) = B_0 \cdot G_0 &= B_0 \left( G_0(0) + \left( G_0(z) - G_0(0) \right) \right) \\
                &= B_0 \left( G_0(0) + B_1 G_1 \right) \\
                &= B_0 \left( G_0(0) + B_1 \left( G_1(0) + \left( G_1(z) - G_1(0) \right) \right) \right) \\
                &= B_0 \left( G_0(0) + B_1 \left( G_1(0) + B_2 G_2 \right) \right) \\
                &= B_0 \left( G_0(0) + B_1 \left( G_1(0) + B_2 \left( G_2(0) + \cdots \right) \right) \right).
            \end{align*}
        \noindent{Substituting $G_i(0) = a_i$ in the factorization yields the formal unwinding series, first described in the PhD thesis of Nahon \cite{Nahon}:}
            \[ F(z) - F(0) = a_0 B_0 + a_1 B_0 B_1 + a_2 B_0 B_1 B_2 + a_3 B_0 B_1 B_2 B_3 \cdots. \]
        \noindent{Although the classic Fourier series works well, it has a variety of limitations: most pressingly, convergence can be numerically slow, especially for discontinuous or multiscale functions. On the other hand, numerical experiments suggest that $r$-Blaschke factorization grants extremely fast convergence (see example in \ref{subsection: example}).}
    
    \subsection{Related Work}
        
        \noindent{It may not be numerically feasible to calculate the roots of any given function $F$. However, findings by Guido and Mary Weiss in \cite{Weiss} allow us to numerically obtain the Blaschke product in a stable manner. The stability of this algorithm has also been investigated by Nahon \cite{Nahon} and Letelier and Saito \cite{Letelier}. Additionally, a number of potential applications of Blaschke factorization have been reviewed, such as electrocardiography signal analysis by Tan, Zhang, and Wu \cite{HauTiengECG}, study of acoustic underwater scattering by Letelier and Saito in \cite{Letelier}, and study of the Doppler Effect by Healy \cite{Doppler}. Furthermore, Qian has extensively studied adaptive Fourier decomposition and its connections to Complex Analysis and Approximation Theory (see \cite{Tao1}, \cite{Tao2}, \cite{Tao3}, \cite{Tao4}, \cite{Tao5}, \cite{Tao6}).}

    \subsection{Example.} \label{subsection: example}
        
        \noindent{Nahon showed in \cite{Nahon} that the formal series converges and is exact after $n$ steps for a polynomial of degree $n$. However, numerical experiments suggest that the unwinding series is a reliable approximation even with fewer than $n$ terms. For instance, consider the polynomial} 
            \[ F(z) = (z - 0.2 - 0.6 i) (z + 0.3 - 0.4 i) (z + 0.5 i) (z - 0.7 + 0.9 i) \]
        \noindent{plotted as $F(e^{it})$ against approximations of $F$ by adding successive terms of the Blaschke unwinding series. \\}
        
    \begin{figure}[h]
        \begin{center} \label{fig: Successive Approximations}
        \begin{minipage}{.4\textwidth}
            \centering
            \includegraphics[width=5.0cm]{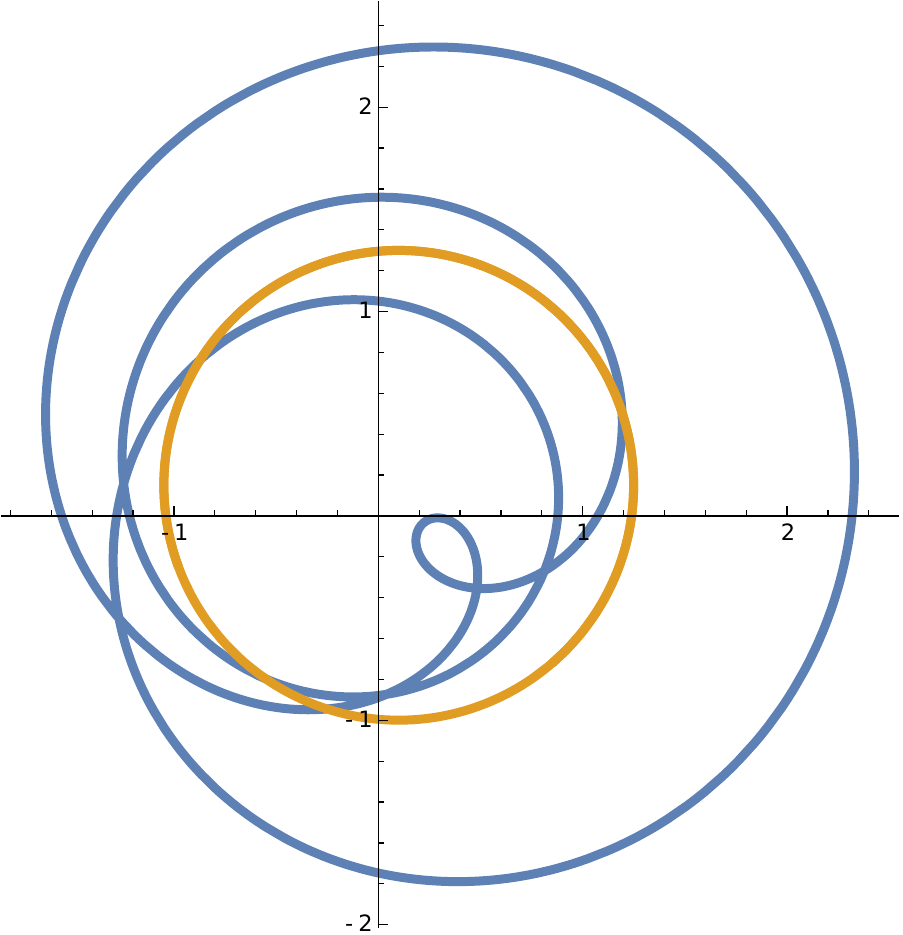}
            \begin{center} $F(0) + a_0B_0$ \end{center}
        \end{minipage}
        \begin{minipage}{.4\textwidth}
            \centering
            \includegraphics[width=5.0cm]{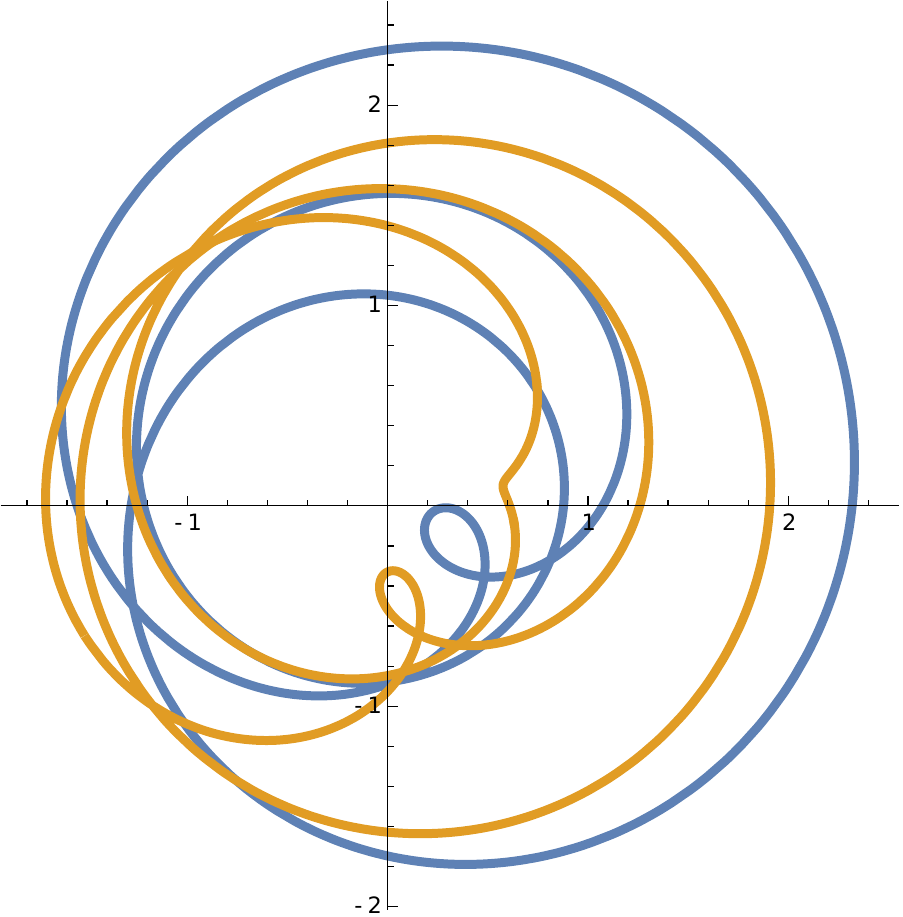}
            \begin{center} $F(0) + a_0B_0 + a_1B_0B_1$ \end{center}
        \end{minipage}
        \end{center}
        \begin{center}
        \begin{minipage}{.4\textwidth}
            \centering
            \includegraphics[width=5.0cm]{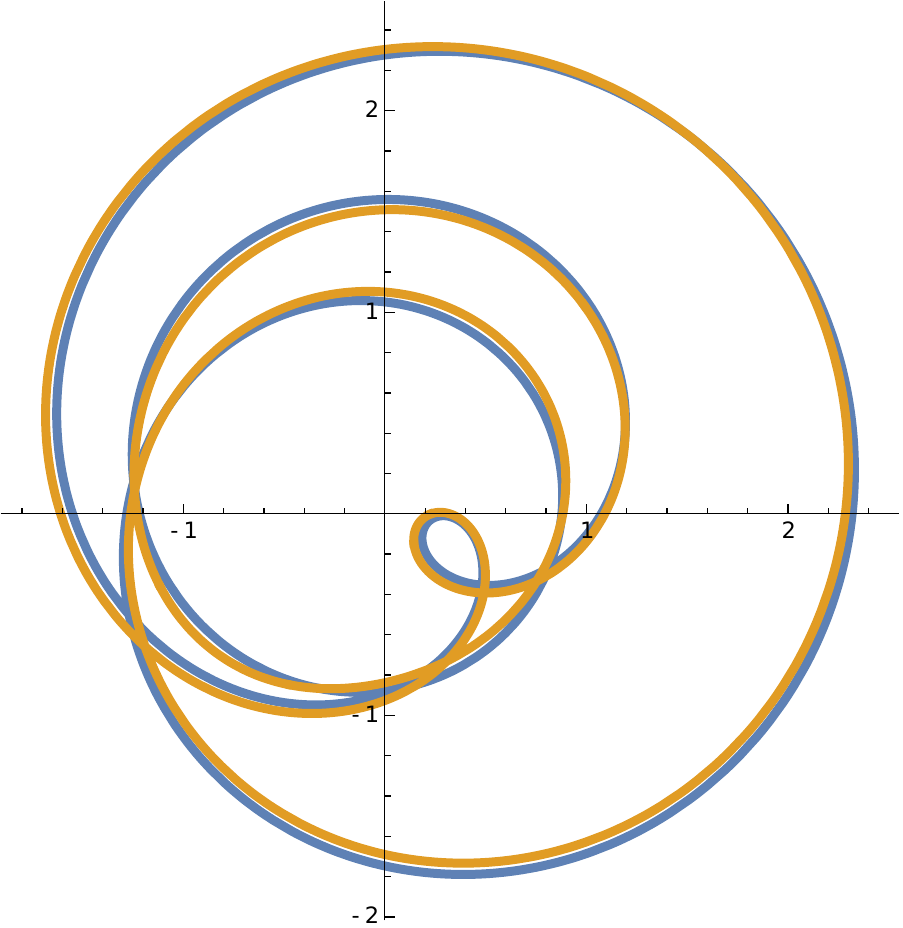}
            \begin{center} $F(0) + a_0B_0 + a_1B_0B_1 + $\\ $a_2 B_0 B_1 B_2$ \end{center}
        \end{minipage}
        \begin{minipage}{.40\textwidth}
            \centering
            \includegraphics[width=5.0cm]{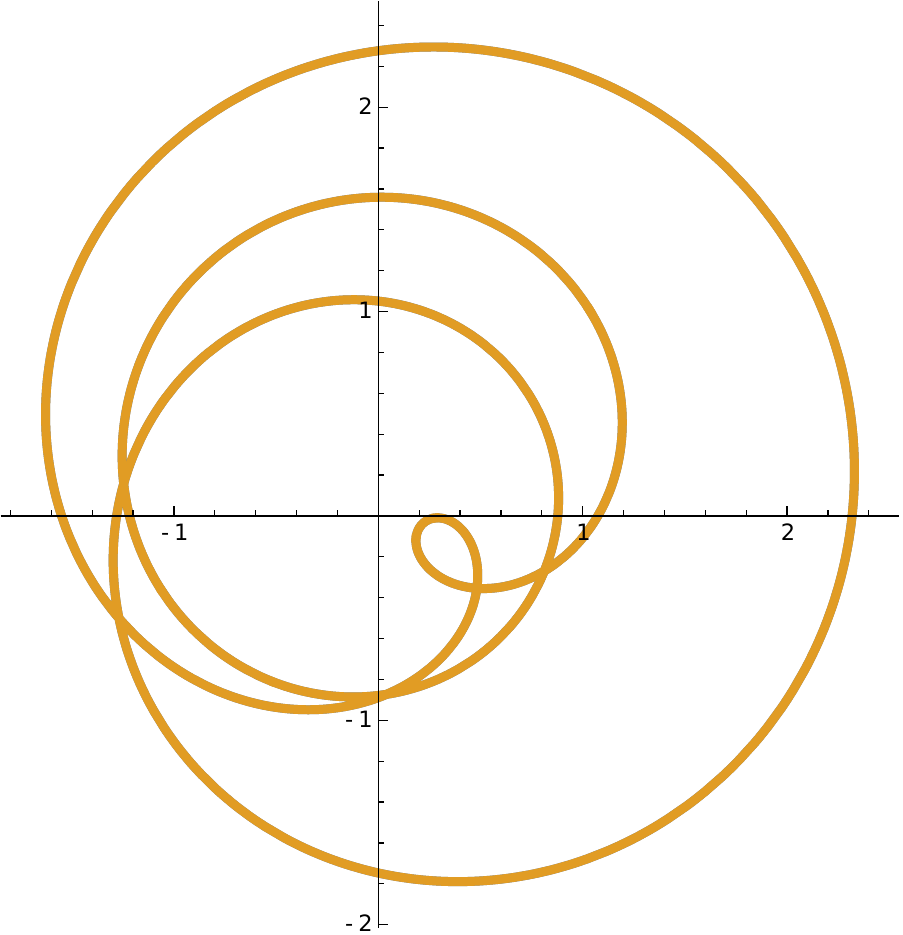}
            \begin{center} $F(0) + a_0B_0 + a_1B_0B_1 + $\\ $a_2 B_0 B_1 B_2 + a_3 B_0 B_1 B_2 B_3$ \end{center}
        \end{minipage}
        \end{center}
    \caption{In blue, we graph the parametrization of $F$ given by $(\textrm{Re}[F(e^{it})],\textrm{Im}[F(e^{it}])$. In orange, we plot approximations of $F$ by adding successive terms of the Blaschke unwinding series for each new illustration.}
    \end{figure}\vspace{3mm}
    
    \noindent{Qualitatively, even three terms of the unwinding series, $F(0) + a_0B_0 + a_1B_0B_1 + a_2B_0B_1B_2$, approximate the function with near-perfect accuracy, even though $\deg(F) = 4$. We now introduce a variation of canonical Blaschke products that improves successive approximations by increasing root capture during factorization.}

\subsection{$\boldsymbol{r}$-Blaschke Factorization}

    \noindent{We now consider factoring out the roots of $G_i(z) - G_i(0)$ contained inside of a disk of radius $r$. Once again, let $F$ be a polynomial in the complex plane. Then, any $F$ can be decomposed as $F(z) - F(0) = B_{r} \cdot G$ where}
        \[ B_r(z) =  \prod_{i \in I} \frac{\left( z - \alpha_i \right) r}{r^2 - \overline{\alpha_i} z} \]
     \noindent{is an $r$-Blaschke product with zeros $\alpha_1,\alpha_2,\ldots$ inside of $\D_r$ and $G$ has no roots inside the disk of radius $r$; that is,}
        \[ G(z) = \prod_{|\alpha_i| > r} (z-\alpha_i) \prod_{0 < |\alpha_i| < r} \frac{1}{r} \left( r^2 - \overline{\alpha_i} z \right). \]
    \noindent{Similarly as before, iterating this new factorization yields the formal variable unwinding series}
        \[ F(z) - F(0) = a_0B_{r_0} + a_1B_{r_0}B_{r_1} + a_2B_{r_0}B_{r_1}B_{r_2} + \cdots. \]
    \noindent{It is important to motivate the construction of $G(z)$. In $r$-Blaschke factorization, all roots $\alpha_i \in \D_r$ are inverted across the boundary $\partial \D_r$ when $r$ is sufficiently large (we say the roots are \textit{captured}). Formally, this root movement is described by}
        \[ \alpha_i \mapsto \frac{r^2}{\overline{\alpha_i}} \]
    \noindent{when $|\alpha_i| < r$. Thus, for each root $\alpha_i$ in the Blaschke product $B_r(z)$, a new root $r^2/\overline{\alpha_i}$ is appended to $G(z)$. In the canonical $1$-Blaschke case, all roots $\alpha_i \in \D$ are inverted across the boundary of the unit disk to form new roots $1/\overline{\alpha_i}$. In the following figure, the left-hand illustration depicts roots captured in $\D$ while the right-hand illustration depicts these same roots captured in $\D_r$.} 
    
    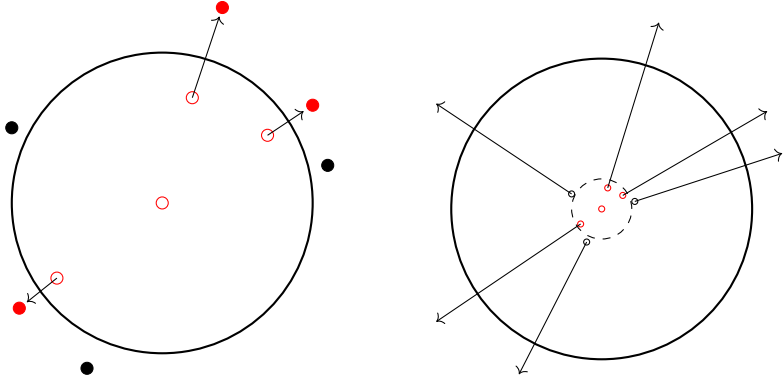
\begin{figure}[h]
            \centering
            \begin{tikzpicture}[scale=2.0]

            \draw [thick] (0,0) circle (1cm);
            
                \filldraw [black] (-1/2,-1.1) circle (0.04cm);
                \filldraw [black] (-1,1/2) circle (0.04cm);
                \filldraw [black] (1.1,0.25) circle (0.04cm);
            
                \draw [red] (0.2,0.7) circle (0.04cm);
                \filldraw [red] (0.4,1.3) circle (0.04cm);
                \draw [->] (0.2,0.7) -- (0.38,1.24);
                
                \draw[red] (0,0) circle (0.04cm);
                
                \draw [red] (0.7,0.45) circle (0.04cm);
                \filldraw [red] (1.0,0.65) circle (0.04cm);
                \draw [->] (0.7,0.45) -- (0.94,0.61);
                
                \draw [red] (-0.7,-0.5) circle (0.04cm);
                \filldraw [red] (-0.95, -0.7) circle (0.04cm);
                \draw [->] (-0.7,-0.5) -- (-0.9, -0.66);
            \end{tikzpicture}
            \begin{tikzpicture}[scale=1.0]
                \draw[white] (-0.5,0) circle (0.05cm);
                \draw[white] (0.5,0) circle (0.05cm);
            \end{tikzpicture}
            \begin{tikzpicture}[scale=2.0]
            \draw [thick] (0,0) circle (1cm);
            \draw [dashed] (0,0) circle (0.2cm);
            
                \draw [black] (-1/10,-0.22) circle (0.02cm);
                \draw [black] (-0.2,1/10) circle (0.02cm);
                \draw [black] (0.22,0.05) circle (0.02cm);
            
                \draw [red] (0.04,0.14) circle (0.02cm);
                \draw [red] (0.14,0.09) circle (0.02cm);
                \draw [red] (-0.14,-0.1) circle (0.02cm);
                \draw [red] (0,0) circle (0.02cm);
                
                \draw [->] (0.04,0.14) -- (0.38,1.24);
                \draw [->] (0.14,0.09) -- (1.1,0.65);
                \draw [->] (-0.14,-0.1) -- (-1.1, -0.75);
                
                \draw [->] (-1/10,-0.22) -- (-0.55,-1.1);
                \draw [->] (-0.2,1/10) -- (-1.1,0.7);
                \draw [->] (0.22,0.05) -- (1.20, 0.37);
            \end{tikzpicture}
        \caption{Variable-Blaschke obtains exponential convergence by increasing root capture. As roots shrink, ie. $\alpha_i \rightarrow 0$, the function behaves as $z^n$ and converges more quickly.}
        \end{figure} \label{img: Shrinking Roots}

\noindent{Additionally, note that the factor $1/r$ in $G(z)$ is used to preserve unit norm on the boundary $\partial \D_r$ so that}
    \[ \left| B_r(r) \right| = \prod_{i \in I} \left| \frac{\left( r - \alpha_i \right) r}{r^2 - \overline{\alpha_i} r} \right| = \prod_{i \in I} \left| \frac{r - \alpha_i}{r - \overline{\alpha_i}} \right| = 1. \]

\section{Statement of Main Results} \label{sec: Statement of Results}

\subsection{A Scaling Symmetry}

    \noindent{The main goal of this paper is to show how Variable-Blaschke can afford exponential convergence in the unwinding series. To begin, we develop various equivalences between $r$-Blaschke factorization on a general polynomial and canonical Blaschke factorization on a scaled version of that polynomial.} 
    
    \begin{proposition}[Blaschke and $r$-Blaschke Equivalence] \label{Blaschke-Blaschke Equivalency}
        \noindent{\textit{Let the polynomials $F(z)$ and $F_{\lambda}(z)$ be given by}}
            \[F(z)=\prod_{i=1}^{n}(z-\alpha_i) \textrm{\hspace{15mm}} F_{\lambda}(z) = \prod_{i=1}^{n} \left( z - \frac{\alpha_i}{\lambda} \right)\]
        \noindent{\textit{and let $\lambda \in \mathbb{R}$ satisfy both $\lambda > 1$ and $\max \{ |\alpha_i| \} < \sqrt{\lambda}$. Then, the Blaschke factorizations $F_{\lambda}(z)=B_1(z) \cdot G_{\lambda}(z)$ and $F(z)=B_{\sqrt{\lambda}}(z) \cdot G(z)$ satisfy}}
            \[ \sqrt{\lambda^n} \cdot G_{\lambda}(z) = G(z).\]
        \textit{\noindent{Alternatively, without any assumptions regarding the norms of the roots, the Blaschke factorizations $F_{\lambda}(z)=B_1(z) \cdot G_{\lambda}(z)$ and $F(z)=B_{\lambda}(z) \cdot G(z)$ satisfy}}    
            \[\lambda^n \cdot G_{\lambda}(z)=G(\lambda z).\]
    \end{proposition}
        
        \label{thm: ShrunkRoots}
        
    \noindent{This equation provides an equivalence between canonical Blaschke factorization and $r$-Blaschke factorization. Additionally, it formalizes a notion paramount to Variable-Blascke factorization: shrinking the roots of a function is equivalent to performing Blaschke on a larger radius, as depicted in the following figure.}
    
     \label{fig: Scaling Equivalency}
    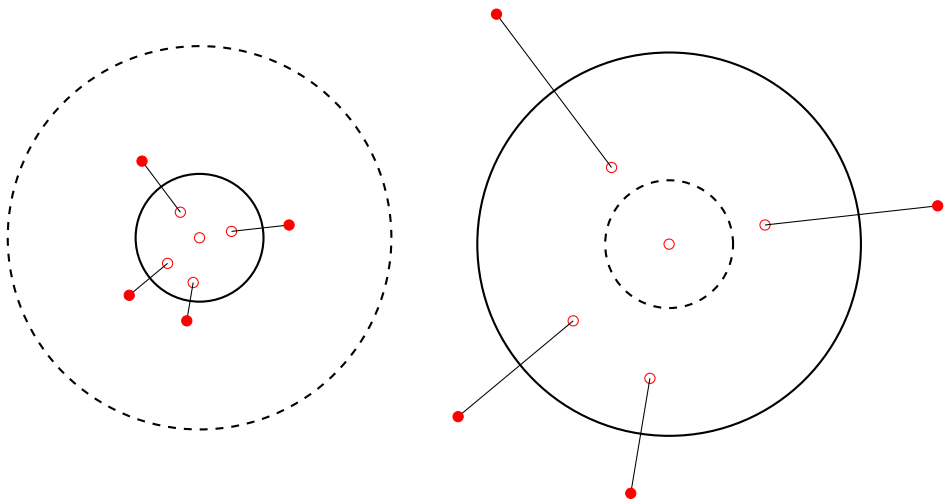
\begin{figure}[h]
            \centering
            \begin{tikzpicture}[scale=0.85]
                \draw[thick] (0,0) circle (1.0cm);
                \draw[thick,dashed] (0,0) circle (3.0cm);
                
                \draw[red] (0,0) circle (0.08cm);
                
                \draw[red] (0.5,0.1) circle (0.08cm);
                \draw[red] (-0.3,0.4) circle (0.08cm);
                \draw[red] (-0.1,-0.7) circle (0.08cm);
                \draw[red] (-0.5,-0.4) circle (0.08cm);
                
                \draw[] (0.5,0.1) -- (1.4,0.2);
                \draw[] (-0.3,0.4) -- (-0.9,1.2);
                \draw[] (-0.1,-0.7) -- (-0.2,-1.3);
                \draw[] (-0.5,-0.4) -- (-1.1,-0.9);
                
                \filldraw[red] (1.4,0.2) circle (0.08cm);
                \filldraw[red] (-0.9,1.2) circle (0.08cm);
                \filldraw[red] (-0.2,-1.3) circle (0.08cm);
                \filldraw[red] (-1.1,-0.9) circle (0.08cm);
                
                
                \draw[white] (0,-4.0) circle (0.08cm);
            \end{tikzpicture}
            \begin{tikzpicture}[scale=0.5]
                \draw[white] (-0.5,0) circle (0.05cm);
                \draw[white] (0.5,0) circle (0.05cm);
            \end{tikzpicture}
            \begin{tikzpicture}[scale=0.85]
                \draw[thick,dashed] (0,0) circle (1.0cm);
                \draw[thick] (0,0) circle (3.0cm);
                
                \draw[red] (0,0) circle (0.08cm);
                
                \draw[red] (1.5,0.3) circle (0.08cm);
                \draw[red] (-0.9,1.2) circle (0.08cm);
                \draw[red] (-0.3,-2.1) circle (0.08cm);
                \draw[red] (-1.5,-1.2) circle (0.08cm);
                
                \draw[] (1.5,0.3) -- (4.2,0.6);
                \draw[] (-0.9,1.2) -- (-2.7,3.6);
                \draw[] (-0.3,-2.1) -- (-0.6,-3.9);
                \draw[] (-1.5,-1.2) -- (-3.3,-2.7);
                
                \filldraw[red] (4.2,0.6) circle (0.08cm);
                \filldraw[red] (-2.7,3.6) circle (0.08cm);
                \filldraw[red] (-0.6,-3.9) circle (0.08cm);
                \filldraw[red] (-3.3,-2.7) circle (0.08cm);
                
            \end{tikzpicture}
        \caption{(Left) Canonical Blaschke factorization performed on the roots $\alpha_i/\lambda$ versus (Right) Variable-Blaschke factorization for the radius $r = \sqrt{\lambda}$ performed on the roots $\alpha_i$. The filled circles represent the radius over which Blaschke is performed; the dashed circles are included to show the radii relative to one another.}
        \end{figure}
    
\subsection{Dirichlet Space}

     \noindent{We specify that exponential convergence occurs in the Dirichlet space $\mathcal{D}$. We define the Dirichlet space on the domain $\Omega \subseteq \C$ as the subspace of the Hardy space $\mathcal{H}^2(\Omega)$ containing all functions}
        \[ f(z) = \sum_{n = 0}^{\infty} a_n z^n \textrm{\hspace{5mm}} (f \in \mathcal{D}) \]
    \noindent{for which the Dirichlet norm}
        \[ \| f(z) \|_{\mathcal{D}}^2 = \|f(z) - f(0) \|_{\mathcal{D}}^2 = \sum_{n=1}^{\infty} n |a_n|^2 \]
    \noindent{is finite. The Dirichlet norm is a useful metric because it is insensitive to constants. If the norms of successive $G_i(z)$ in a regular iterative Blaschke decomposition satisfy}
        \[ \left\| G_{i}(z) - G_i(0) \right\|_{\mathcal{D}}^2 \leq \frac{1}{2} \left\| G_{i+1}(z) \right\|_{\mathcal{D}}^2, \]
    \noindent{then this implies exponential decay. In Theorem \ref{cor: Radii Sequence}, we discuss parameters for obtaining this successive decay in norm. Although we restrict our study to the Dirichlet space, our results can likely be extended to more general subspaces of $\mathcal{H}^2(\mathbb{T})$ and we believe this to be an interesting avenue for future research.}
    
\subsection{Contraction Inequalities}
    \begin{theorem} \label{thm: ExponentialConvergence}
        \noindent{Let the polynomials $F(z)$ and $F_{\lambda}(z)$ be given by}
            \[F(z)=\prod_{i=1}^{n}(z-\alpha_i) \textrm{\hspace{15mm}} F_{\lambda}(z) = \prod_{i=1}^{n} \left( z - \frac{\alpha_i}{\lambda} \right)\]
        \noindent{and let} 
            \[ \lambda \geq 6.15 \left( \max_{1 \leq i \leq n} |\alpha_i| \right). \]
        \noindent{Then, the Blaschke factorization $F_{\lambda}(z)=B(z) \cdot G_{\lambda}(z)$}
        \noindent{satisfies}
            \[ \| G_{\lambda}(z) \|_{\mathcal{D}}^2 \leq \frac{1}{2} \| F_{\lambda}(z) \|_{\mathcal{D}}^2. \]
    \end{theorem}
    
    \noindent{This Theorem demonstrates that exponential convergence occurs in the Dirichlet space if the roots are sufficiently scaled down. Alternatively, one can obtain the same inequality when the $n$-th power mean of the scaled roots is sufficiently small, as described by the following Theorem.}
    
   \begin{theorem}\label{prop:RootDistribution}  
        \noindent{Let $F(z)$ and $F_{\lambda}(z)$ be defined as above. If}
      \[ \lambda \geq 6.75\sqrt[n]{\frac{1}{n}\sum_{i=1}^n |\alpha_i|^n}\]
      \noindent{then, as above,}
        \[ \| G_{\lambda}(z) \|_{\mathcal{D}}^2 \leq \frac{1}{2} \| F_{\lambda}(z) \|_{\mathcal{D}}^2. \]
    \end{theorem}
    
    \noindent{Note that, if $n \rightarrow \infty$, then we recover Theorem \ref{thm: ExponentialConvergence} with a worse constant. The first result shows that it suffices to scale the roots to lie in some small disk, while the second result provides some flexibility with the distribution of the roots. Essentially, Theorems 2.1 and 2.2 quantify the intuition that the closer a function is to $z^n$, the faster the unwinding series converges. One might ask whether we can significantly decrease the constant 6.15. To that end, we examined the function $F(z) = (z-m)^n$ and determined the value $m=m_0$ for which the inequality}
	   \[ \| G(z) \|_{\mathcal{D}}^2 \leq \frac{1}{2} \| F(z) \|_{\mathcal{D}}^2 \]
	\noindent{is an equality. We numerically computed $m_0$ for large $n$ since the inequality breaks for $m > m_0$. While we have been unable to prove so, numerical experiments suggest that $m_0$ approaches $1/2$ as $n$ approaches infinity. This would imply that the constant 6.15 cannot be replaced with 2.}
	
\subsection{Exponential Convergence}

    \begin{proposition}[$\lambda$-Blaschke] \label{thm: Vlad}
        \noindent{Let $F$ be a monic polynomial. Then, for some sufficiently large $\lambda>0$, the Blaschke factorization $F(z) = B_{\lambda}(z) \cdot G(z)$ satisfies:}
            \[ \left\| \frac{1}{\lambda^n} \cdot G(z) \right\|_{\mathcal{D}}^2 \leq \frac{1}{2} \left\| F(z) \right\|_{\mathcal{D}}^2. \]
    \end{proposition}

    \noindent{One might think it strange that a factor of $\lambda^{-n}$ appears in the inequality. However, whereas $|B(z)|=1$ on the boundary of the unit disk, $B_{\lambda}$ is of size $\lambda^{-n}$ (for large $\lambda$). Thus, in the actual term of the unwinding series, $B_{\lambda}$ serves to rescale back the $G(z)$ term, which is why allowing a factor of $\lambda^{-n}$ here is justified. This means that we can pick some sequence of $\lambda$ such that the contraction inequality holds at every step. In turn, we can pick some sequence of radii corresponding to the sequence of $\lambda$ with which to perform $r$-Blaschke.}
    
    \begin{theorem}[Exponential Convergence] \label{cor: Radii Sequence}
        For every polynomial $F$, there is a sequence of radii $r_1,\ldots,r_n$ that guarantees the above exponential convergence at each iteration of $r$-Blaschke factorization.
    \end{theorem}
    
    \begin{figure}[h] \label{img: Root Translation}
        \centering
        \centering
        \begin{tikzpicture}[scale=1.3] 
        
            \node[right] at (0.75,0) (a) {$r_1$};
            \node[right] at (1.65,0) (a) {$r_2$};
            \node[right] at (2.5,0) (a) {$r_3$};
        
            \draw [thick] (0,0) circle (0.75cm);
            \draw [thick] (0,0) circle (1.65cm);
            \draw [thick] (0,0) circle (2.5cm);
            
            \draw [black] (0,0) circle (0.04cm);
            
            \draw [blue] (0.1, -0.1) circle (0.04cm);
            \draw [->, dashed] (0.11, -0.14) -- (0.80, -1.1);
            \filldraw [blue] (0.80, -1.1) circle (0.04cm);
            \draw[->, dashed] (0.80, -1.1) -- (-0.85,-0.64)  node[midway,left,below=0.0cm,rotate=-17] {$+G_0(0)$};
            
            \draw [blue] (-0.90,-0.65) circle (0.04cm);
            \draw [->, dashed] (-0.90,-0.65) -- (-1.8, -1.4);
            \filldraw [blue]  (-1.85, -1.45) circle (0.04cm);
            \draw [->, dashed] (-1.85, -1.45) -- (0.46, -1.90) node[midway,left,below=0.0cm,rotate=-10] {$+G_1(0)$};
            
            \draw[blue] (0.46, -1.90) circle (0.04cm);
            \draw [->, dashed] (0.46, -1.90) -- (0.76,-2.90);
            \filldraw[blue] (0.78,-2.94) circle (0.04cm);
            
            \draw [red] (0.4,0.35) circle (0.04cm);
            \draw [->, dashed] (0.4,0.35)--(0.7,0.65);
            \filldraw [red] (0.74,0.69) circle (0.04cm);
            \draw [->, dashed] (0.74,0.69) -- (-0.6,1.15) node[midway,left,above=0.0cm,rotate=-20] {$+G_0(0)$};
            
            \draw [red] (-0.66,1.17) circle (0.04cm);
            \draw [->,dashed] (-0.66,1.17) -- (-0.96,1.88);
            \filldraw[red] (-0.99,1.92) circle (0.04cm);
            \draw [->, dashed] (-0.99,1.92) -- (1.25, 1.54) node[midway,left,above=0.0cm,rotate=-10] {$+G_1(0)$};
            
            \draw [red] (1.3, 1.55) circle (0.04cm);
            \draw [->, dashed] (1.3, 1.55) -- (1.95, 2.30);
            \filldraw [red] (1.99, 2.34) circle (0.04cm);
        \end{tikzpicture}
    \caption{After a root is captured, it is translated by $G_i(0)$ and a new root is appended to the origin. We can pick a sequence of $r_i$ to optimize this capture-translation process.}
    \end{figure}
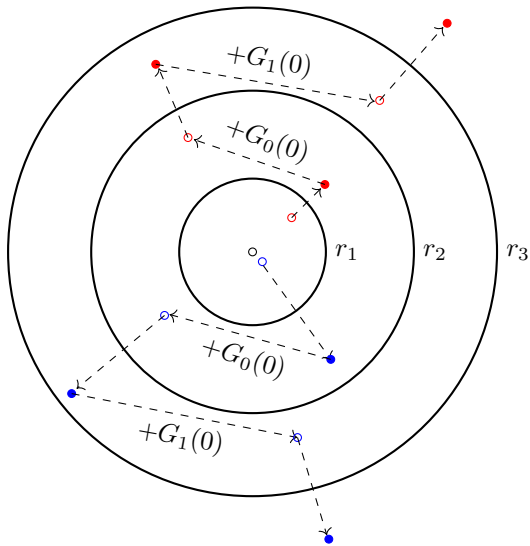
    
    \noindent{Explicitly, our minimal condition is to select a radius $r_i$ such that}
        \[ r_i > \max \{ |\alpha_{i}| : G_i(\alpha_i) - G_i(0) = 0 \}. \]
    \noindent{Equivalently, our radius $r_{i+1}$ should at least be large enough to capture all of the roots generated from the $i$th iteration of Blaschke. Additionally, $r_i$ may need to be increased to fit the parameters for exponential convergence described in Theorem \ref{thm: ExponentialConvergence}. Thus, the exact radii selected depend on root translation, since the roots of $B_{r_{i+1}}$ will be the roots of $G_{i}(z) - G_{i}(0)$ translated by $G_{i+1}(0)$. Equivalently, after inverting a collection $\alpha_i$ over a disk boundary, the newfound roots will all be affected by the translation by $G_{i+1}(0)$, as shown above in Figure 4. \\}
        
    \noindent{However, we lack the ability to predict where the roots of $G_{i+1}(z) - G_{i+1}(0)$ will be located. This is a trade-off when working with nonlinear functions and poses difficulty for general proofs concerning Blaschke behavior. In some cases, we are a priori able to bound the roots of $G_{i+1}(z).$ Indeed, according to a result by Ostrowski in \cite{Book}, given roots $\alpha_i$ of $G_{i}(z)$, we can order roots $\beta_i$ of $G_{i}(z)-G_{i}(0)$ such that}
        \[ |\alpha_i-\beta_i| < 2n \left( \prod_i^n |\alpha_i| \right)^{1/n}. \]
    \noindent{If this bound guarantees that any nonzero $\beta_i$ has norm at least $\varepsilon_i$ for some $\varepsilon_i > 0$, then we are able to bound the norms of the roots of $G_{i+1}(z)$ above by $r_i^2/\varepsilon_i$. According to the results of Theorem 2.1, we can thus pick}
        \[ r_{i+1} \ge 6.15 \left( \frac{r_i^2}{\varepsilon_i} \right). \]
        
\section{Computational Experiments}

    \subsection{Motivation.}

    \noindent{Computational examples suggested that the quality of unwinding series approximation under $r$-Blaschke improves as $r$ increases. \\}
    
    \begin{figure}[h]
    \begin{center} \label{Variable-Blaschke Example}
        \begin{minipage}{.30\textwidth}
            \includegraphics[width=4.0cm]{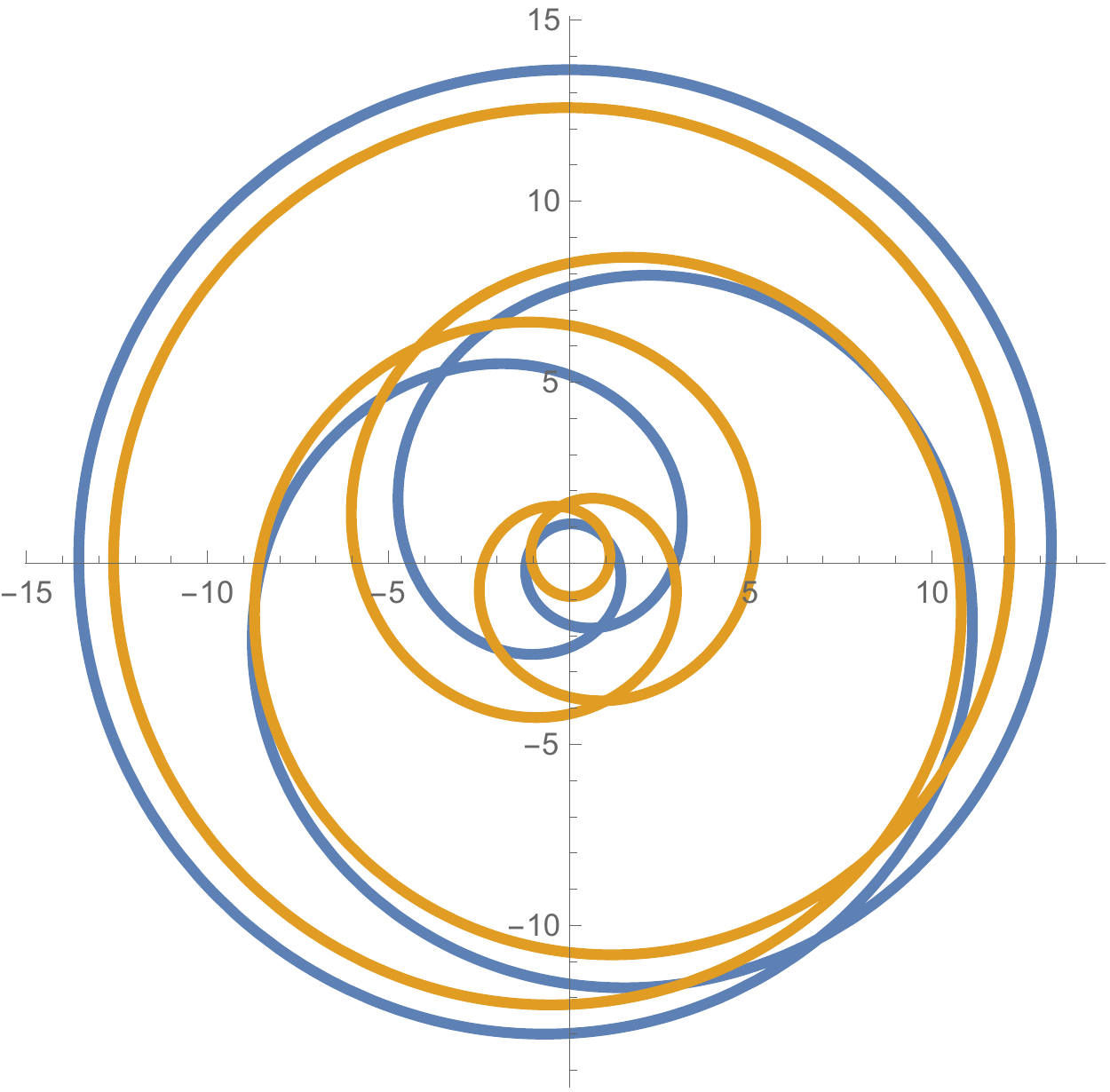}
            \begin{center} $F(0) + a_0 (B_1)_0 + a_1 (B_1)_0 (B_1)_1$ \end{center}
        \end{minipage}
        \begin{minipage}{.30\textwidth}
            \includegraphics[width=4.0cm]{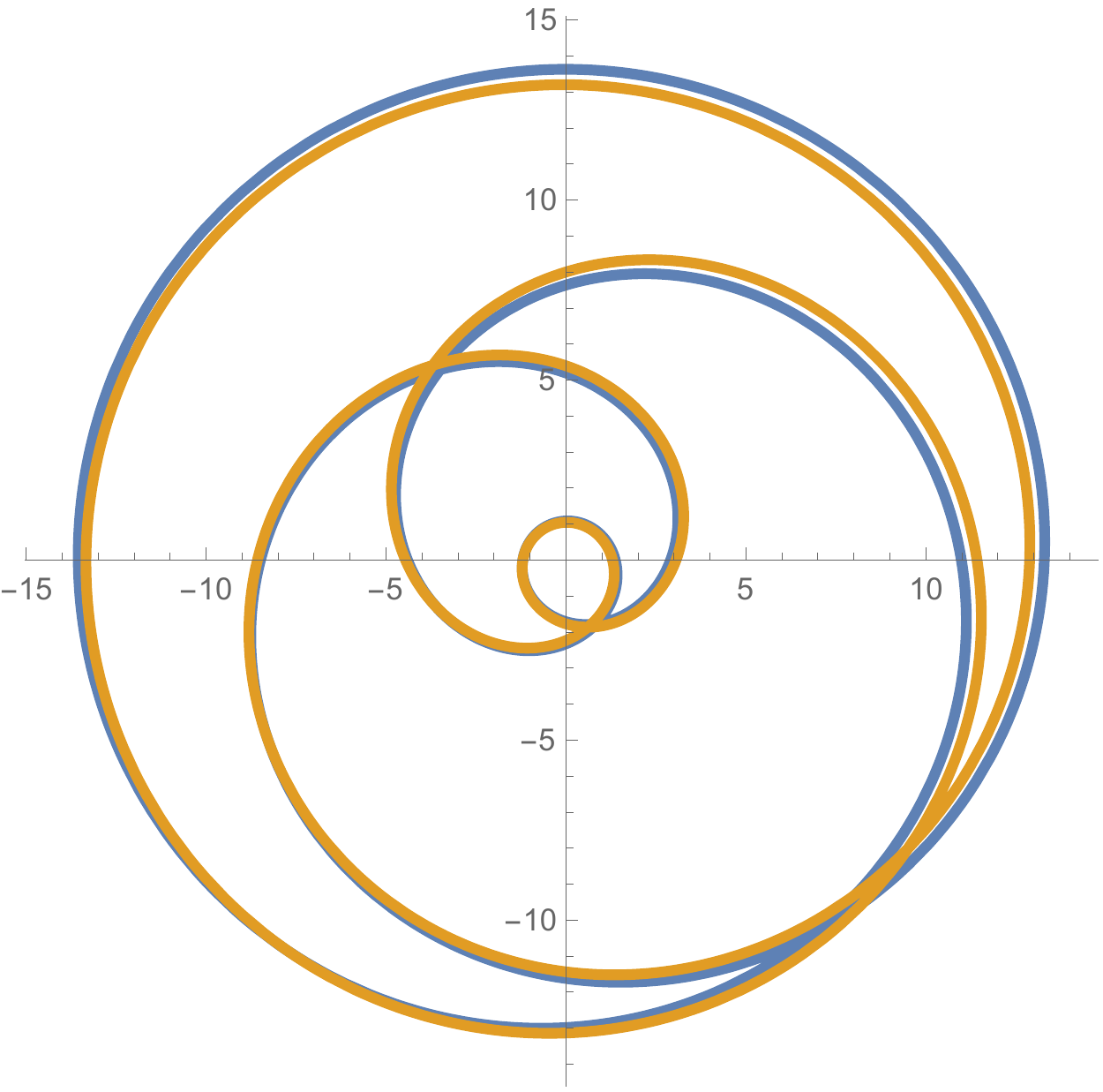}
            \begin{center} $F(0) +a_0 (B_2)_0 + a_1 (B_2)_0 (B_2)_1$ \end{center}
        \end{minipage}
        \begin{minipage}{.30\textwidth}
            \includegraphics[width=4.0cm]{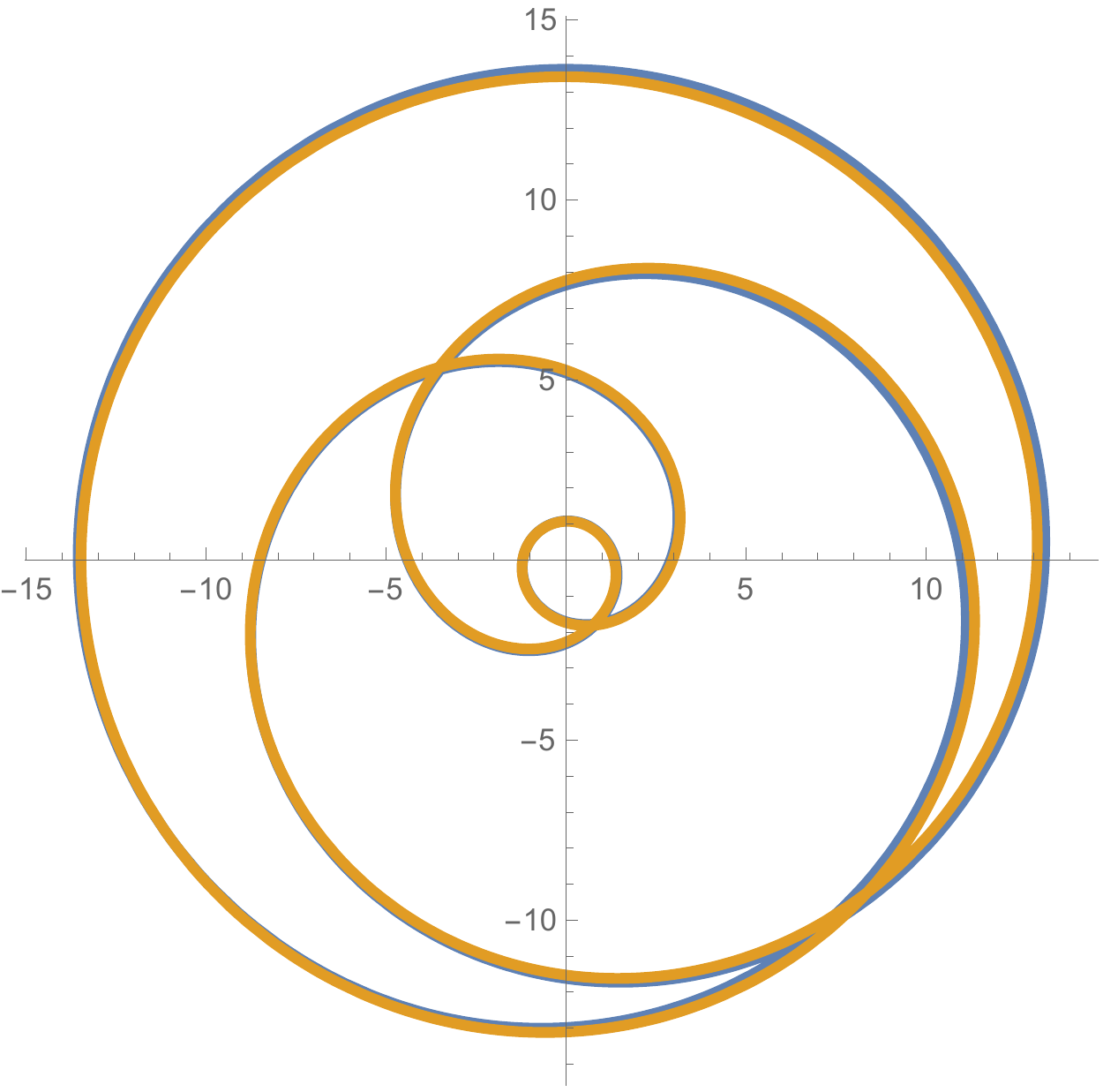}
            \begin{center}$F(0) +a_0 (B_3)_0 + a_1 (B_3)_0 (B_3)_1$ \end{center}
        \end{minipage}
    \end{center}
    \caption{$B_i$ can denote both the $i$th Blaschke product in the unwinding series and the Variable-Blaschke product with radius equal to $i$. When it is absolutely necessary to convey both pieces of information (as in this case), we use the notation $(B_r)_i$ to refer to the $i$-th $r$-Blaschke product in the unwinding series.}
    \end{figure}
    
    \noindent{Indeed, for every polynomial we have tested, faster convergence occurs for larger radii $r$. For example, consider the function $F(z) = (z - 0.5 i) (z - 0.2 - 0.6 i) (z + 0.1 + 0.3 i) (z - 0.3 i) (z - 6)$. In blue, we plotted the parametrization of $F$ given by $(\textrm{Re}[F(e^{it})],\textrm{Im}[F(e^{it}])$. In orange, we plotted the second degree unwinding series approximation using the Blaschke products $B$, $B_2$, and $B_3$.}
    
    \subsection{Method.}

    \noindent{In seeking quantitative evidence for exponential convergence, we adopted the following method:}
        \begin{enumerate}
            \item Construct a polynomial from $n$ random roots $\alpha_i$ from a uniform distribution in a disk of radius $m$:
                \[ F(z) = \prod_{i=1}^{n} (z-\alpha_i) \textrm{\hspace{10mm}} \alpha_i \in \mathcal{U}(\D_m). \]
            \item Fix some radius $r_1$ with which to perform $r$-Blaschke factorization.
            \item Calculate the error for various successive approximations of $F(z)$ via
                \begin{align*}
                    \textrm{Error} &:= \left\| F(z) - F(0) - \sum_{i=0}^{L} a_i \prod_{i=0}^{L} B_{r_i} \right\|_{L^2(\partial \D)} \\
                    &= \int_{0}^{2\pi} \left| F ( e^{it} ) - F(0) - \sum_{i=0}^{L} a_i \prod_{i=0}^{L} B_{r_i} \right|^2 dt.
                \end{align*}
                \noindent{where the unwinding series contains $L+1$ terms.}
            \item Repeat this error calculation for other radii $r_2,r_3,\ldots,r_k$.
            \item Repeat the entire process for other random polynomials of degree $n$ and average error results.
        \end{enumerate}
        
    \subsection{Discussion of Results.}
        
    \noindent{We computed error calculations across a variety of different circumstances. For example, consider $r$-Blaschke on a 15th degree polynomial whose roots are distributed uniformly in $\D$, $\D_5$, and $\D_{25}$, using the $r$-Blaschke products $B_1$, $B_2$, $B_3$, $B_4$, and $B_5$, plotted on logarithmic graphs:}
    
    \begin{figure}[h]
        \centering
        \begin{tikzpicture}[y=.07cm, x=0.5cm, scale = 1.0]
         	
         	\foreach \y in {-70,-30,-20,-10,0,10}
         	    \draw[line width=0.0cm,gray] (0,\y)--(15,\y);
         	\foreach \y in {-60,-50,-40}
         	    \draw[line width=0.0cm,gray] (0,\y)--(1,\y);
         	\foreach \y in {-60,-50,-40}
         	    \draw[line width=0.0cm,gray] (4,\y)--(15,\y);
         	\foreach \x in {2,3}
         	    \draw[line width=0.0cm,gray] (\x,-30)--(\x,10);
         	\foreach \x in {1,4,5,6,7,8,9,10,11,12,13,14,15}
         	    \draw[line width=0.0cm,gray] (\x,-70)--(\x,10);
         	
        	\draw[->] (0,0) -- coordinate (x axis mid) (15.5,0);
            	\draw[<->] (0,-73) -- coordinate (y axis mid) (0,13);
            	\foreach \x in {1,...,15}
             		\draw (\x,1pt) -- (\x,-3pt)
        			node[anchor=north] {};
            	\foreach \y in {-70,-60,-50,-40,-30,-20,-10,0,10}
             		\draw (1pt,\y) -- (-3pt,\y) 
             			node[anchor=east] {\y};
             
            \filldraw[thin,red] (1,5.49935) circle (0.05cm);
            \filldraw[thin,red] (2,5.37006) circle (0.05cm);
            \filldraw[thin,red] (3,	5.02404) circle (0.05cm);
            \filldraw[thin,red] (4,	4.41028) circle (0.05cm);
            \filldraw[thin,red] (5,	3.58982) circle (0.05cm);
            \filldraw[thin,red] (6,	2.80219) circle (0.05cm);
            \filldraw[thin,red] (7,	1.90963) circle (0.05cm);
            \filldraw[thin,red] (8,	0.276982) circle (0.05cm);
            \filldraw[thin,red] (9,	-2.24568) circle (0.05cm);
            \filldraw[thin,red] (10,-4.2031) circle (0.05cm);
            \filldraw[thin,red] (11,-6.18324) circle (0.05cm);
            \filldraw[thin,red] (12,-8.6395) circle (0.05cm);
            \filldraw[thin,red] (13,-12.3938) circle (0.05cm);
            \filldraw[thin,red] (14,-16.225) circle (0.05cm);
            \draw[thin,red,dashed] (1,5.49935)--(2,5.37006)--(3,	5.02404)--(4,	4.41028)--(5,	3.58982)--(6,	2.80219)--(7,	1.90963)--(8,	0.276982)--(9,	-2.24568)--(10,-4.2031)--(11,-6.18324)--(12,-8.6395)--(13,-12.3938)--(14,-16.225);
            
            \filldraw[thin,blue] (1,4.43602) circle (0.05cm);
            \filldraw[thin,blue] (2,2.42053) circle (0.05cm);
            \filldraw[thin,blue] (3,-0.332903) circle (0.05cm);
            \filldraw[thin,blue] (4,-3.61327) circle (0.05cm);
            \filldraw[thin,blue] (5,-7.21322) circle (0.05cm);
            \filldraw[thin,blue] (6,-10.601) circle (0.05cm);
            \filldraw[thin,blue] (7,-14.0277) circle (0.05cm);
            \filldraw[thin,blue] (8,-18.4868) circle (0.05cm);
            \filldraw[thin,blue] (9,-25.2332) circle (0.05cm);
            \filldraw[thin,blue] (10,-28.3522) circle (0.05cm);
            \filldraw[thin,blue] (11,-32.3611) circle (0.05cm);
            \filldraw[thin,blue] (12,-37.9509) circle (0.05cm);
            \filldraw[thin,blue] (13,-45.4189) circle (0.05cm);
            \filldraw[thin,blue] (14,-51.0898) circle (0.05cm);
            \draw[thin,blue,dashed] (1,4.43602)--(2,2.42053)--(3,-0.332903)--(4,-3.61327)--(5,-7.21322)--(6,-10.601)--(7,-14.0277)--(8,-18.4868)--(9,-25.2332)--(10,-28.3522)--(11,-32.3611)--(12,-37.9509)--(13,-45.4189)--(14,-51.0898);
            
            \foreach \y [count = \yi] in {3.2448,	-0.22594,	-4.53192,	-9.40004,	-14.6287,	-19.6381,	-24.6047,	-30.5986,	-39.1964,	-42.5318,	-48.0987,	-55.3639,	-63.0799,	-63.332}
                    \filldraw[thin,teal] (\yi,\y) circle (0.05cm);
            \draw[thin,teal,dashed] (1,3.2448)--(2,-0.22594)--(3,-4.53192)--(4,-9.40004)--(5,-14.6287)--(6,-19.6381)--(7,-24.6047)--(8,-30.5986)--(9,-39.1964)--(10,-42.5318)--(11,-48.0987)--(12,-55.3639)--(13,-63.0799)--(14,-63.332);
            
            \foreach \y [count = \yi] in {2.25683,	-2.30424,	-7.73837,	-13.7432,	-20.125,	-26.2886,	-32.3777,	-39.4916,	-49.2769,	-53.7702,	-60.4152,	-63.4561,	-63.4512,	-63.4511}
                \filldraw[thin,orange] (\yi,\y) circle (0.05cm);
            \draw[thin,orange,dashed] (1,2.25683)--(2,-2.30424)--(3,-7.73837)--(4,-13.7432)--(5,	-20.125)--(6,-26.2886)--(7,-32.3777)--(8,-39.4916)--(9,-49.2769)--(10,-53.7702)--(11,-60.4152)--(12,-63.4561)--(13,-63.4512)--(14,-63.4511);
            
            \foreach \y [count = \yi] in {1.44189,	-3.98436,	-10.3001,	-17.1916,	-24.4668,	-31.525,	-38.4938,	-46.4856,	-57.1742,	-62.2506,	-63.5486,	-63.5315,	-63.5317,	-63.5317}
                \filldraw[thin,purple] (\yi,\y) circle (0.05cm);
            \draw[thin,purple,dashed] (1,1.44189)--(2,-3.98436)--(3,	-10.3001)--(4,	-17.1916)--(5,	-24.4668)--(6,	-31.525)--(7,	-38.4938)--(8,	-46.4856)--(9,	-57.1742)--(10,	-62.2506)--(11,	-63.5486)--(12,	-63.5315)--(13,	-63.5317)--(14,	-63.5317);
            
            
            \begin{scope}[shift={(1.6,-62)}] 
        	\draw 
        		plot[mark=*, mark options={fill=purple}] (0.25,0)
        		node[right=0.1cm]{$B_5$};
        	\draw[yshift=\baselineskip]
        		plot[mark=*, mark options={fill=orange}] (0.25,0)
        		node[right=0.1cm]{$B_4$};
        	\draw[yshift=2\baselineskip]
        		plot[mark=*, mark options={fill=teal}] (0.25,0)
        		node[right=0.1cm]{$B_3$};
        	\draw[yshift=3\baselineskip]
        		plot[mark=*, mark options={fill=blue}] (0.25,0)
        		node[right=0.1cm]{$B_2$};
        	\draw[yshift=4\baselineskip]
        		plot[mark=*, mark options={fill=red}] (0.25,0)
        		node[right=0.1cm]{$B_1$};
        	\end{scope}
             
        	\node[] at (7.5,15) {(A): Roots in $\D$};
        	\node[] at (7.5,-75) {Iteration $(L)$};
        	\node[rotate=90, above=0.8cm] at (y axis mid) {$\log$(Error)};
            \node[] at (18,0) {};
        \end{tikzpicture}
    \end{figure}
    \begin{figure}[h]
        \begin{tikzpicture}[y=.07cm, x=0.5cm,scale = 1.0]
         	\foreach \y in {-40,10,20,30,40}
         	    \draw[line width=0.0cm,gray] (0,\y)--(15,\y);
         	\foreach \x in {1,4,5,6,7,8,9,10,11,12,13,14}
         	    \draw[line width=0.0cm,gray] (\x,-40)--(\x,40);
         	    
         	\draw[line width=0.0cm,gray] (2,10)--(2,40);
         	\draw[line width=0.0cm,gray] (3,10)--(3,40);
         	
         	\draw[line width=0.0cm,gray] (15,-40)--(15,40);
         	
     	    \draw[line width=0.0cm,gray] (0,-20)--(1,-20);
     	    \draw[line width=0.0cm,gray] (4,-20)--(15,-20);
     	    \draw[line width=0.0cm,gray] (0,-10)--(1,-10);
     	    \draw[line width=0.0cm,gray] (4,-10)--(15,-10);
     	    \draw[line width=0.0cm,gray] (0,-30)--(1,-30);
     	    \draw[line width=0.0cm,gray] (4,-30)--(15,-30);
     	    
     	    \draw[line width=0.0cm,gray] (2,0)--(2,10);
     	    \draw[line width=0.0cm,gray] (3,0)--(3,10);
         	        
        	\draw[->] (0,0) -- coordinate (x axis mid) (15.5,0);
            	\draw[<->] (0,-43) -- coordinate (y axis mid) (0,43);
        		\foreach \x in {1,...,15}
             		\draw (\x,1pt) -- (\x,-3pt)
        			node[anchor=north] {};
        		\draw(1,1pt)--(1,-3pt)node[anchor=north] {};
            	\foreach \y in {-40,-30,-20,-10,0,10,20,30,40}
             		\draw (1pt,\y) -- (-3pt,\y) 
             			node[anchor=east] {\y};
             
            \foreach \y [count = \yi] in {33.5791,	30.9497,	29.1704,	26.8096,	25.5418,	23.3236,	20.7583,	19.0954,	16.9413,	14.4748,	11.375,	8.0504,	4.59914,	0.0428439}
                \filldraw[thin,red] (\yi,\y) circle (0.05cm);
                \draw[thin,red,dashed] (1,33.5791)--(2,30.9497)--(3,	29.1704)--(4,	26.8096)--(5,	25.5418)--(6,	23.3236)--(7,	20.7583)--(8,19.0954)--(9,	16.9413)--(10,	14.4748)--(11,	11.375)--(12,8.0504)--(13,	4.59914)--(14,	0.0428439);
            
            \foreach \y [count = \yi] in {32.5848,29.2369,	27.1766,	24.882,	22.4489	,20.2919,	17.6786,	15.1891,	12.3271,	9.38457,	6.38709,	2.67671,-0.956078,	-6.03351}
                \filldraw[thin,blue] (\yi,\y) circle (0.05cm);
                \draw[thin,blue,dashed] (1,32.5848)--(2,29.2369)--(3,	27.1766)--(4,	24.882)--(5,	22.4489)--(6,	20.2919)--(7,	17.6786)--(8,	15.1891)--(9,	12.3271)--(10,	9.38457)--(11,	6.38709)--(12,	2.67671) --(13,-0.956078)--(14,	-6.03351);
            
            \foreach \y [count = \yi] in {30.5053,	27.4253,	24.461,	21.9915,	19.2673,	16.4481,	13.5243,	10.6886,	7.62,	4.26995,	0.517411,	-3.63826,	-8.39567,	-15.3551}
                    \filldraw[thin,teal] (\yi,\y) circle (0.05cm);
            \draw[thin,teal,dashed] (1,30.5053)--(2,27.4253)--(3,	24.461)--(4,	21.9915)--(5,	19.2673)--(6,	16.4481)--(7,	13.5243)--(8,	10.6886)--(9,	7.62)--(10,	4.26995)--(11,	0.517411)--(12,	-3.63826)--(13,	-8.39567)--(14,	-15.3551);
            
            \foreach \y [count = \yi] in {29.853,	26.09,	22.9764,	19.9546,	16.5845,	13.0325, 9.57852,	5.9132,	1.91871,	-2.38644,	-7.02806,	-11.0494,	-15.4667,	-22.498}
                \filldraw[thin,orange] (\yi,\y) circle (0.05cm);
            \draw[thin,orange,dashed] (1,29.853)--(2,26.09)--(3,	22.9764)--(4,	19.9546)--(5,16.5845)--(6,	13.0325)--(7, 9.57852)--(8,	5.9132)--(9,	1.91871)--(10,	-2.38644)--(11,	-7.02806)--(12,	-11.0494)--(13,	-15.4667)--(14,	-22.498);
            
            \foreach \y [count = \yi] in {29.3066,	26.0727,	22.5097,18.3948	,14.4314,	9.9758,	5.34516,	0.932817,	-3.87698,	-8.81503,	-13.7915,	-20.2271,	-24.3463,	-32.0367}
                \filldraw[thin,purple] (\yi,\y) circle (0.05cm);
            \draw[thin,purple,dashed] (1,29.3066)--(2,	26.0727)--(3,	22.5097)--(4,18.3948)--(5	,14.4314)--(6,	9.9758)--(7,	5.34516)--(8,	0.932817)--(9,	-3.87698)--(10,	-8.81503)--(11,	-13.7915)--(12,	-20.2271)--(13,	-24.3463)--(14,	-32.0367);
            
            
            \begin{scope}[shift={(1.6,-33)}] 
        	\draw 
        		plot[mark=*, mark options={fill=purple}] (0.25,0)
        		node[right=0.1cm]{$B_5$};
        	\draw[yshift=\baselineskip]
        		plot[mark=*, mark options={fill=orange}] (0.25,0)
        		node[right=0.1cm]{$B_4$};
        	\draw[yshift=2\baselineskip]
        		plot[mark=*, mark options={fill=teal}] (0.25,0)
        		node[right=0.1cm]{$B_3$};
        	\draw[yshift=3\baselineskip]
        		plot[mark=*, mark options={fill=blue}] (0.25,0)
        		node[right=0.1cm]{$B_2$};
        	\draw[yshift=4\baselineskip]
        		plot[mark=*, mark options={fill=red}] (0.25,0)
        		node[right=0.1cm]{$B_1$};
        	\end{scope}
             
        	\node[] at (7.5,45) {(B): Roots in $\D_5$};
        	\node[] at (7.5,-45) {Iteration $(L)$};
        	\node[rotate=90, above=0.8cm] at (y axis mid) {$\log$(Error)};
            \node[] at (18,0) {};
        \end{tikzpicture}
    \end{figure}
    \begin{figure}[h]
        \begin{tikzpicture}[y=.07cm, x=0.5cm, scale = 1.0]
         	
         	\foreach \y in {40,50,60,70,80}
         	    \draw[line width=0.0cm,gray] (0,\y)--(15,\y);
         	\foreach \y in {10,20,30}
         	    \draw[line width=0.0cm,gray] (0,\y)--(1,\y);
         	\foreach \x in {1,4,5,6,7,8,9,10,11,12,13,14,15}
         	    \draw[line width=0.0cm,gray] (\x,0)--(\x,80);
         	    
         	\draw[line width=0.0cm,gray] (2,40)--(2,80);
         	\draw[line width=0.0cm,gray] (3,40)--(3,80);
         	
         	\foreach \y in {10,20,30,40}
         	    \draw[line width=0.0cm,gray] (4,\y)--(15,\y);
         	
        	\draw[->] (0,0) -- coordinate (x axis mid) (15.3,0);
            	\draw[<->] (0,-3) -- coordinate (y axis mid) (0,83);
        		\foreach \x in {1,...,15}
             		\draw (\x,1pt) -- (\x,-3pt)
        			node[anchor=north] {};
            	\foreach \y in {0,10,20,30,40,50,60,70,80}
             		\draw (1pt,\y) -- (-3pt,\y) 
             			node[anchor=east] {\y};
             
            \foreach \y [count = \yi] in {70.4981,	67.3922,	63.5783	,59.2728,	55.4928	,51.3124,	46.5379,	41.3116	,35.5935,	29.4654	,23.3041,	16.8712,	9.76235	,2.35012}
                \filldraw[thin,red] (\yi,\y) circle (0.05cm);
                \draw[thin,red,dashed] (1,70.4981)--(2,	67.3922)--(3,	63.5783)--(4	,59.2728)--(5,	55.4928)--(6	,51.3124)--(7,	46.5379)--(8,	41.3116)--(9	,35.5935)--(10,	29.4654)--(11	,23.3041)--(12,	16.8712)--(13,	9.76235)--(14	,2.35012);
            
            \foreach \y [count = \yi] in {70.4949,	67.3773,	63.5459,	59.2048,	55.4376	,51.2739,	46.5114,	41.2814,	35.5401,	29.3718	,23.1732,	16.7335	,9.61428,	2.26984}
                \filldraw[thin,blue] (\yi,\y) circle (0.05cm);
                \draw[thin,blue,dashed] (1,70.4949)--(2,	67.3773)--(3,	63.5459)--(4,	59.2048)--(5,	55.4376)--(6	,51.2739)--(7,	46.5114)--(8,	41.2814)--(9,	35.5401)--(10,	29.3718)--(11	,23.1732)--(12,	16.7335)--(13	,9.61428)--(14,	2.26984);
            
            \foreach \y [count = \yi] in {70.2656,	67.2749,	63.368,	59.0239	,55.3156,	51.2066	,46.4599,	41.2169,	35.4444,	29.2095	,22.9604,	16.5047,	9.33015	,2.13341}
                    \filldraw[thin,teal] (\yi,\y) circle (0.05cm);
            \draw[thin,teal,dashed](1,70.2656)--(2,	67.2749)--(3,	63.368)--(4,	59.0239)--(5	,55.3156)--(6,	51.2066)--(7	,46.4599)--(8,	41.2169)--(9,	35.4444)--(10,	29.2095)--(11	,22.9604)--(12,	16.5047)--(13,	9.33015)--(14	,2.13341);
            
            \foreach \y [count = \yi] in {69.9117,	66.713,	62.8291,	58.4944,	54.4057,	50.1267,	45.3272,	40.0761,	34.2924,	28.0644,	21.8557,	15.4451,	8.32098,	1.95388}
                \filldraw[thin,orange] (\yi,\y) circle (0.05cm);
            \draw[thin,orange,dashed] (1,69.9117)--(2,	66.713)--(3,	62.8291)--(4,	58.4944)--(5,	54.4057)--(6,	50.1267)--(7,	45.3272)--(8,	40.0761)--(9,	34.2924)--(10,	28.0644)--(11,	21.8557)--(12,	15.4451)--(13,	8.32098)--(14,	1.95388);
            
            \foreach \y [count = \yi] in {69.6667,	66.247,	62.333,	58.1637,	53.894,	49.3562,	44.4224,	39.1245,	33.3135	,27.0985,	20.9502,	14.5873,	7.46279, 1.55584}
                \filldraw[thin,purple] (\yi,\y) circle (0.05cm);
            \draw[thin,purple,dashed] (1,69.6667)--(2,	66.247)--(3,	62.333)--(4,	58.1637)--(5,	53.894)--(6,	49.3562)--(7,	44.4224)--(8,	39.1245)--(9,	33.3135)--(10	,27.0985)--(11,	20.9502)--(12,	14.5873)--(13,	7.46279)--(14, 1.55584);
            
            
            \begin{scope}[shift={(1.6,8)}] 
        	\draw 
        		plot[mark=*, mark options={fill=purple}] (0.25,0)
        		node[right=0.1cm]{$B_5$};
        	\draw[yshift=\baselineskip]
        		plot[mark=*, mark options={fill=orange}] (0.25,0)
        		node[right=0.1cm]{$B_4$};
        	\draw[yshift=2\baselineskip]
        		plot[mark=*, mark options={fill=teal}] (0.25,0)
        		node[right=0.1cm]{$B_3$};
        	\draw[yshift=3\baselineskip]
        		plot[mark=*, mark options={fill=blue}] (0.25,0)
        		node[right=0.1cm]{$B_2$};
        	\draw[yshift=4\baselineskip]
        		plot[mark=*, mark options={fill=red}] (0.25,0)
        		node[right=0.1cm]{$B_1$};
        	\end{scope}
             
        	\node[] at (7.5,85) {(C): Roots in $\D_{25}$};
        	\node[] at (7.5,-8) {Iteration $(L)$};
        	\node[rotate=90, above=0.8cm] at (y axis mid) {$\log$(Error)};
        	\node[] at (18,0) {};
        \end{tikzpicture}
    \caption{Averaged $\log$(Error) from 100 samples of degree 15 polynomials with $\alpha_i \in \mathcal{U}(\D), \alpha_i \in \mathcal{U}(\D_{5}),$ and $\alpha_i \in \mathcal{U}(\D_{25})$, respectively. Note that the approximation becomes exact at $L=n$; thus, $\log$(Error) results for $L=n$ are omitted. Additionally, note that $\log$(Error) results for $B_4$ and $B_5$ in the top-left graph become too small for accurate numerical precision and thus flatten out.}
    \end{figure}
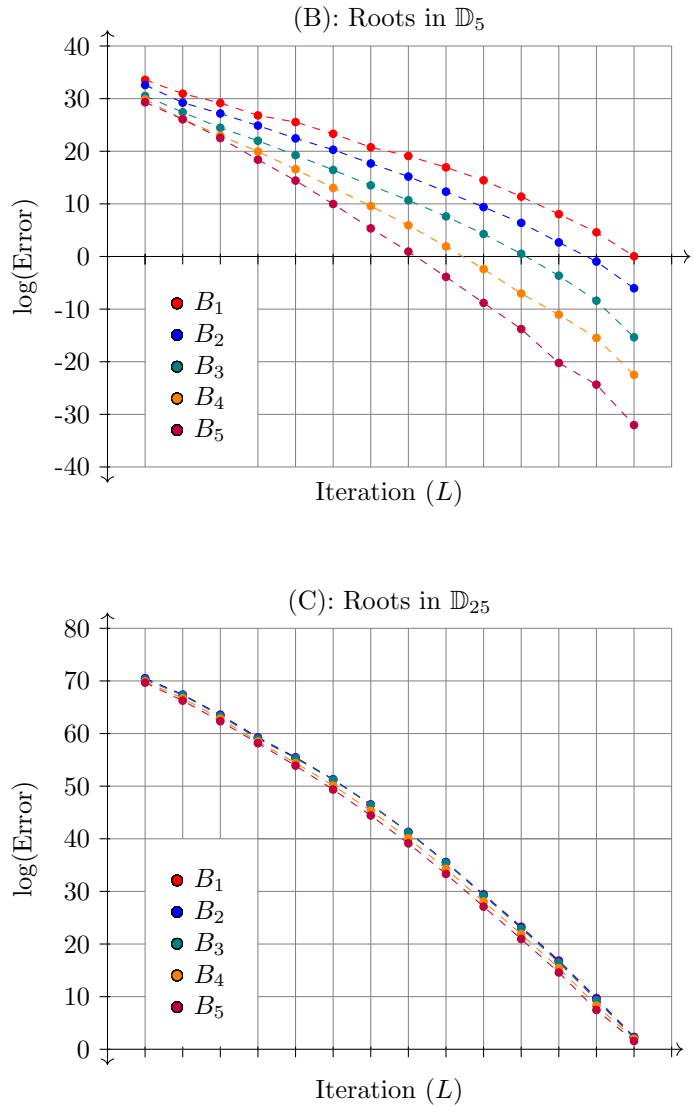

    \noindent{Firstly, we see that the error of approximation significantly increases (and potentially scales) as the distribution of roots is increased. Consider the $\log$(Error) of canonical Blaschke factorization on averaged 15th degree polynomials whilst modulating the distribution of the roots.} \pagebreak
    
    \begin{figure}[h]
        \centering
        \begin{tabular}{||c||c|c|c|c|c|c|c||}
            \hline \hline
            $L$ & 1 & 2 & 3 & 4 & 5 & 6 & 7 \\ 
            \hline \hline
            $\alpha_i \in \D$ & 5.49935&	5.37006&	5.02404&	4.41028&	3.58982	&2.80219&	1.90963 \\
            $\alpha_i \in \D_5$ & 33.5791&	30.9497&	29.1704&	26.8096	&25.5418&	23.3236&	20.7583 \\
            $\alpha_i \in \D_{25}$ & 70.4981&	67.3922&	63.5783&	59.2728&	55.4928	&51.3124&	46.5379 \\
          \hline \hline
        \end{tabular}
    \caption{Averaged $\log$(Error) from 100 samples of canonical Blaschke factorization on degree 15 polynomials with distributed roots $\alpha_i \in \mathcal{U}(\D), \alpha_i \in \mathcal{U}(\D_{5}),$ and $\alpha_i \in \mathcal{U}(\D_{25})$, respectively.}
    \end{figure}

    \noindent{We see that the error increases as the roots become more widely distributed. This trend holds across all successive iterations $L$ and for all radii $r$ with which we perform $r$-Blaschke factorization. Secondly, qualitatively, spacing differs between Variable-Blaschke plots depending on the distribution of the roots. This relationship between root distribution and Blaschke error is given by Proposition \ref{thm: ShrunkRoots}. Clearly, both increasing the number of successive iterations given some fixed $r$ and increasing the radius $r$ with which Variable-Blaschke is performed decrease error of approximation. However, we have not been able to develop any concrete determination of how quickly error decreases or the exact role that root distribution plays.}

\section{An Additional Result}

    \noindent{This section discusses an additional contraction result that carries over from \cite{CoifmanStefan} to the more general case.}

    \begin{theorem}[Generalized Operator Equivalence] \label{ScalingCorollary}
        \noindent{Suppose that $\max \{|\alpha_i|\} < \gamma< \lambda$ where $\gamma,\lambda \in \R$ and that $F_\gamma(z)$ and $F_\lambda(z)$ are defined as before. Then, $G_{\lambda}(z)$ and $G_{\gamma}(z)$, defined by the Variable-Blaschke factorizations}
            \[F_{\lambda}(z)=B_{\gamma}(z)\cdot G_{\lambda}(z)\]
        \noindent{and}
            \[F_{\gamma}(\gamma z/\lambda)=B_{\lambda}(z)\cdot G_{\gamma}(z)\]
        \noindent{satisfy}
            \[G_{\gamma}(z)=G_{\lambda}(z).\]
        \noindent{Alternatively, $G_{\lambda}(z)$ and $G_{\gamma}(z)$, defined by the Variable-Blaschke factorizations}
            \[F_{\gamma}(z)=B_{\lambda}(z)\cdot G_{\gamma}(z)\]
        \noindent{and}
            \[F_{\lambda}(z)=B_{\gamma}(z)\cdot G_{\lambda}(z)\]
        \noindent{satisfy}
            \[\gamma^n G_{\gamma}(\lambda z) = \lambda^n G_{\lambda}(\gamma z).\]
    \end{theorem}
    
    \noindent{While Corollary \ref{Blaschke-Blaschke Equivalency} provided an equivalence between canonical Blaschke factorization and $r$-Blaschke factorization, this Theorem provides an equivalence between two $r$-Blaschke factorizations for different radii $\lambda$ and $\gamma$. Indeed, this statement is equivalent to Proposition \ref{Blaschke-Blaschke Equivalency} if $\gamma = 1$. Most importantly, it tells us that one can always explicitly relate $\gamma$-Blaschke and $\lambda$-Blaschke using scalar constants. The following Proposition is adapted from \cite{BlaschkeSurvey} to work with the $r$-Blaschke framework.}
    
    \begin{proposition} \label{def: LogDeriv}
    For $\C \setminus \{ \alpha_i, r^2/\overline{\alpha_i} : 1 \leq i \leq n \}$, the logarithmic derivative of an $r$-Blaschke product is
        \[ \frac{B_r'(z)}{B_r(z)} = \sum_{j=1}^{n} \frac{r^2 - |\alpha_j|^2}{(r^2 - \overline{\alpha_j}z) (z-\alpha_j)}. \]
    \noindent{Therefore, we have}
        \[ \left| B_r'(re^{it}) \right| = \sum_{j=1}^{n} \frac{r^2 - |\alpha_j|^2}{r \left| e^{it} -\alpha_j \right|^2}. \]
    \end{proposition}

    \begin{theorem}[One-Step $r$-Blaschke] \label{OneStepR}
        \textit{Let $F$ be holomorphic on a neighborhood of the unit disk and $\alpha \in \C$ with $|\alpha|<r$. If}
            \[ f = (z-\alpha) F \textit{\hspace{5mm} and \hspace{5mm}} g = \frac{1}{r} \left( r^2 - \overline{\alpha}z \right) F \]
        \textit{then}
            \[ \int_{0}^{2\pi} \left| g'(re^{it}) \right|^2 \: dt \leq \int_{0}^{2\pi} \left| f'(re^{it}) \right|^2 \: dt - \left( 1 - \frac{|\alpha|^2}{r^2} \right) \int_{0}^{2\pi} \left| F(re^{it}) \right|^2 \: dt \]
        \textit{whenever all terms are defined and finite.}
    \end{theorem}
    
    
\section{Proofs}\label{sec: Proofs}

    \noindent{In this section, we outline proofs for all Theorems and Propositions:}
    
    \begin{enumerate}
        \item Theorem 2.1
        \item Theorem 2.2
        \item Proposition 2.2 ($\lambda$-Blaschke)
        \item Theorem 4.1 (Generalized Operator Equivalence)
        \item Proposition 2.1 (Blaschke and $r$-Blaschke Equivalence)
        \item Proposition 4.1
        \item Theorem 4.2 (One-Step $r$-Blaschke) \\
    \end{enumerate}
    
    \noindent{We begin by establishing two lemmas for a couple of small facts regarding the coefficients of the polynomial $F_{\lambda}$. Recall that}
        \[ F_{\lambda}(z) = \prod_{i=1}^{n} \left( z - \frac{\alpha_i}{\lambda} \right). \]

    \begin{lemma}
    \noindent{Denote $c_k$ as the coefficient of $z^k$ of $F_\lambda$. Then}
       \[ |c_k|  \leq \binom{n}{k}\left(\frac{\varepsilon}{\lambda}\right)^{n-k}.\]
       
    \noindent{The above helps us write our results in terms of scaling all the roots down to fit within a certain small disk. Alternatively, a more robust result states}
     \[ |c_k| \leq \binom{n}{k} \cdot \left(\frac{1}{n}\sum_{i=1}^n |\alpha_i/\lambda|^n\right)^{1-k/n}. \]

   \noindent{Thus, we require only the $n$-th power mean of the norms of the roots to be small for the bound to hold.}
   \end{lemma}
    \begin{proof}
      \noindent{Observe}
       \[ |c_k| = \Bigg| \sum_{|S| = n-k} \left[ \prod_{i \in S} \frac{\alpha_i}{\lambda} \right] \Bigg|\leq \sum_{|S| = n-k} \left|\prod_{i\in S}\frac{\alpha_i}{\lambda}\right| \leq \binom{n}{k}\left(\frac{\varepsilon}{\lambda}\right)^{n-k}.\]
       \noindent{To prove the second bound, we use Muirhead's inequality, which states}
       \[\sum_{\text{sym}}x_1^{a_1}\ldots x_n^{a_n} \ge \sum_{\text{sym}} x_1^{b_1}\ldots x_n^{b_n}\]
       \noindent{for nonnegative integers $a_1,\ldots,a_n,b_1,\ldots,b_n,x_1,\ldots,x_n,$ where $(a_1,\ldots,a_n)$ majorizes $(b_1,\ldots,b_n).$ In particular, the inequality holds for $(a_1,\ldots,a_n) = (n-k,0,\ldots,0)$ and $(b_1,\ldots,b_n) = (1,\ldots,1,0,\ldots,0),$ where there are $n-k$ 1's. We also use the Power Mean inequality, which states}
       \[\sqrt[a]{\frac{x_1^a+\ldots+x_n^a}{n}} \ge \sqrt[b]{\frac{x_1^b+\ldots+x_n^b}{n}}\]
       \noindent{for nonnegative integers $x_1,\ldots,x_n,$ and nonnegative integers $a \ge b.$ Now observe}
            \[ |c_k| = \Bigg| \sum_{|S| = n-k} \left[ \prod_{i \in S} \frac{\alpha_i}{\lambda} \right] \Bigg| \leq \sum_{|S| = n-k} \left|\prod_{i\in S}\frac{\alpha_i}{\lambda}\right|\]
        
    \noindent{and use Muirhead's inequality and the Power Mean inequality to obtain}
    
     \[ |c_k| \leq \sum_{i=1}^n |\alpha_i/\lambda|^{n-k} \cdot {n \choose k}/n\leq \binom{n}{k} \cdot \left(\frac{1}{n}\sum_{i=1}^n |\alpha_i/\lambda|^n\right)^{1-k/n}. \]
    \end{proof}
    \begin{lemma}
    \noindent{\textit{Without loss of generality, assume the leading coefficient $c_n$ of $F_\lambda$ is 1. Recall}}
        \[ F_\lambda (z) = \prod_{i=1}^{n} \left( z - \frac{\alpha_i}{\lambda} \right) = \sum_{k=0}^{n} c_k z^k \textrm{\hspace{5mm} and \hspace{5mm}} G_\lambda (z) = \prod_{i=1}^{n} \left(1-\frac{\overline{\alpha_i}}{\lambda}\right) = \sum_{k=0}^n d_k z^k. \]
   \noindent{\textit{Then, the coefficients $c_k$ of $F_\lambda$ and $d_k$ of $G_\lambda$ satisfy}}
       \[c_k = \overline{d_{n-k}}.\]
    \end{lemma}
    \begin{proof}
   
   \noindent{Let $H$ be the polynomial with coefficients $d_k$ satisfying the above relation; we will show $G_\lambda = H.$ Indeed, note that as all roots of $F_\lambda$ are captured via Blaschke, all nonzero roots $\alpha_i$ are sent to $1/\overline{\alpha_i}$, while its zero roots disappear, and}
        \[G(0) = \prod_{i=1}^n (1-\overline{\alpha_i}(0)) = 1.\]
   \noindent{Observe now that for $z\neq 0$}
        \[\overline{F_\lambda(z)/z^{n}} = H (1/\overline{z})\]
   \noindent{and}
        \[H(0)=d_0=\overline{c_n}=1.\]
   \noindent{Thus, the roots of $H$ are exactly those of $G_\lambda$, and $H$ and $G_\lambda$ are both monic, so $H=G_\lambda$, as desired.}
   
    \end{proof}
    
\subsection{Proof of Theorem \ref{thm: ExponentialConvergence}}
    
    \begin{proof}
    \noindent{The general idea of this proof is as follows. We proceed through a series of approximations using an undetermined parameter $\delta$, to obtain some conditions on $\lambda$, which, if satisfied, will give the theorem's result. Then, we retroactively pick a specific value of $\delta$ such that the demands on $\lambda$ are minimized. If one does not care about the size of $\lambda$, a much shorter proof could be given, but it is a reasonable goal to not use $\lambda$'s which are larger than needed.}\\

    \noindent{\emph{Part 1: Conditions on $\lambda$.} For simplicity, we want all roots of $F_\lambda$ to be captured via Blaschke factorization, so we at least assume $\lambda > \varepsilon$ where $\varepsilon = \max \{ |\alpha_i| \}$. Next, we bound the quantity}
        \[ \frac{1-|\alpha/\lambda|^2}{|z-\alpha/\lambda|^2} \]
    \noindent{on the unit disk from below by $1-\delta$, $0<\delta<1$, with the exact value of $\delta$ to be stipulated later. Since $|z|=1$, we have}
        \[ \frac{1-|\alpha/\lambda|^2}{|z-\alpha/\lambda|^2} \geq 
        \frac{1-|\alpha/\lambda|^2}{(1+|\alpha/\lambda|)^2} \geq
        \frac{1-(\varepsilon/\lambda)^2}{(1+\varepsilon/\lambda)^2}. \]
    \noindent{Hence, it is sufficient to show}
        \[ \frac{1-(\varepsilon/\lambda)^2}{(1+\varepsilon/\lambda)^2}>1-\delta. \]
    \noindent{Solving the inequality for $\lambda$ gives us}
        \[ \lambda> \left( \frac{2-\delta}{\delta} \right) \varepsilon. \]

    \noindent{Let us call this condition on $\lambda$ ``I''. Now, we assume this condition is satisfied. In \cite{CoifmanStefan}, Coifman and Steinerberger showed that}
        \[ \| F_{\lambda} \|_{\mathcal{D}}^2 - \| G_{\lambda} \|_{\mathcal{D}}^2 = \int_{0}^{2\pi} \left| G_{\lambda}(e^{it}) \right|^2 \sum_{i=1}^{n} \frac{1- |\alpha_i|^2}{| z - \alpha_i |^2} \: dt. \]
        
    \noindent{Our goal is thus to show}
        \[ ||F_\lambda||^2_\mathcal{D}=||G_\lambda||^2_\mathcal{D}+\frac{1}{2}\int_{d\mathcal{D}}|G|^2\sum_i\frac{1-|\alpha_i/\lambda|^2}{|z-\alpha_i/\lambda|^2} \geq 2||G_\lambda||^2_\mathcal{D} \]
    \noindent{or} 
         \[ \frac{1}{2}\int_{d\mathcal{D}}|G|^2\sum_i\frac{1-|\alpha_i/\lambda|^2}{|z-\alpha_i/\lambda|^2} \geq ||G_\lambda||^2_\mathcal{D} \]
    \noindent{so, by using the bound from above, it is sufficient to show} 
        \[ \frac{(1-\delta)n}{2}||G_\lambda||^2_{\mathcal{L}^2}\geq||G_\lambda||^2_\mathcal{D}. \]
    \noindent{Writing this in coefficient form, we have}
        \[ \frac{(1-\delta)n}{2}\sum_{k=0}^n|g_k|^2\geq\sum_{k=1}^nk|g_k|^2 \]
    \noindent{so we need}
        \[ \sum_{k=0}^n\left(1-\frac{2k}{(1-\delta)n}\right)|g_k|^2= 1+\sum_{k=1}^n\left(1-\frac{2k}{(1-\delta)n}\right)|g_k|^2\geq0 \]
    \noindent{where $g_k$ are the coefficients of $G_\lambda$. Notice that for}
        \[ k<\frac{(1-\delta)n}{2} \]
    \noindent{the terms in the sum are positive, and for}             
        \[ k>\frac{(1-\delta)n}{2} \]
    \noindent{they are negative. Thus, let}
        \[ c=\left\lceil\frac{(1-\delta)n}{2}\right\rceil. \]
    \noindent{Then, it is sufficient to show}             
        \[ \left|\sum_{k=c}^n\left(1-\frac{2k}{(1-\delta)n}\right)|g_k|^2\right|\leq 1.\]
    \noindent{We find that}
        \begin{align*}
            \sum_{k=c}^n\left|1-\frac{2k}{(1-\delta)n}\right||g_k|^2 &\leq \frac{1+\delta}{1-\delta} \sum_{k=c}^n|g_k|^2 \\
            &=\frac{1+\delta}{1-\delta}\sum_{k=c}^n|c_{n-k}|^2 \\
            &\leq \frac{1+\delta}{1-\delta}\sum_{k=c}^n\left({{{n}\choose{k}}\left(\frac{\varepsilon}{\lambda}\right)^k}\right)^2.
        \end{align*}
    \noindent{We now will bound the sum from above by a geometric series. First, let}
        \[ A(k)={{n}\choose{k}} \left(\frac{\varepsilon}{\lambda}\right)^k \]
    \noindent{and define $U(k)=A(k+1)/A(k)$, such that then $A(k+1)=U(k)A(k)$. By calculation,}
        \[ U(k)= \frac{A(k+1)}{A(k)}=\frac{\varepsilon}{\lambda}\frac{n-k}{k+1}. \]
    \noindent{Notice that as $k$ increases, $U(k)$ decreases. Thus, let us first demand that $U(c)<1$. This means that}
        \[ \lambda>\varepsilon\frac{n-c}{c+1}. \]
    \noindent{Call this condition on $\lambda$ ``II,'' and assume that $\lambda$ satisfies this condition. If so, then we know we can bound $A(k)$, for $n\geq k\geq c$, from above by $A(c)U(c)^{k-c}$. Hence, we have}
        \[ A(k)^2\leq A(c)^2U(c)^{2(k-c)}= {{n}\choose{c}}^2\left(\frac{\varepsilon}{\lambda}\right)^{2c}\left(\frac{\varepsilon}{\lambda}\right)^{2(k-c)}\left(\frac{n-c}{c+1}\right)^{2(k-c)}.\]
    \noindent{Therefore,}
        \begin{align*}
            \frac{1+\delta}{1-\delta}\sum_{k=c}^n\left({{{n}\choose{k}}\left(\frac{\varepsilon}{\lambda}\right)^k}\right)^2 &\leq \frac{1+\delta}{1-\delta}{{n}\choose{c}}^2\left(\frac{\varepsilon}{\lambda}\right)^{2c}\sum_{k=c}^n\left({\left(\frac{\varepsilon}{\lambda}\frac{n-c}{c+1}\right)^{2}}\right)^{(k-c)} \\
            &\leq \frac{1+\delta}{1-\delta}{{n}\choose{c}}^2\left(\frac{\varepsilon}{\lambda}\right)^{2c}\sum_{k=c}^\infty\left({\left(\frac{\varepsilon}{\lambda}\frac{n-c}{c+1}\right)^{2}}\right)^{(k-c)} \\
            &= \frac{1+\delta}{1-\delta}{{n}\choose{c}}^2\left(\frac{\varepsilon}{\lambda}\right)^{2c}\frac{1}{1-\left(\frac{\varepsilon}{\lambda}\frac{n-c}{c+1}\right)^2}.
        \end{align*}
    \noindent{Hence, it is sufficient to show that this quantity is less than 1. Equivalently, it suffices to show}
        \[ \lambda>\varepsilon\frac{\left(\frac{1+\delta}{1-\delta}\right)^{\frac{1}{2c}}{{n}\choose{c}}^{\frac{1}{c}}}{\left(1-\left(\frac{\varepsilon}{\lambda}\frac{n-c}{c+1}\right)^2\right)^{\frac{1}{c}}}. \]
    \noindent{Call this condition on $\lambda$ ``III.'' As of now, this is not an explicit condition on $\lambda$, but will become so after we pick a specific $\delta$. Now, we have three conditions on $\lambda$, which would give the result of the theorem if satisfied. \\}

    \noindent{\emph{Part 2: The choice of $\delta$.} The choice of $\delta$ affects the size of $\lambda$ in a non-trivial way. We shall pick a good value of $\delta$ by considering the behavior of Conditions I, II, and III as $n\rightarrow\infty$ (where the conditions become simplified). Furthermore, since all three conditions take the form of a quantity multiplied by $\varepsilon$, we will ignore this $\varepsilon$ from our calculations. Condition I does not depend on $n$, so let}
        \[ \mbox{I}_\infty(\delta)=\frac{2-\delta}{\delta}. \]
    \noindent{We can first simplify the form of Condition II by noting that} 
        \[ c=\left\lceil\frac{(1-\delta)n}{2}\right\rceil \]
    \noindent{implies}
        \[ c\geq\frac{(1-\delta)n}{2}. \]
    \noindent{Thus, Condition II can be satisfied if}
        \[ \lambda>\varepsilon\cdot\frac{n-\frac{(1-\delta)n}{2}}{\frac{(1-\delta)n}{2}+1}=\varepsilon\cdot\frac{n(1+\delta)}{2+(1-\delta)n}. \]
    \noindent{Consider this our new Condition II. Then,}
        \[ \lim_{n\to\infty}\frac{n(1+\delta)}{2+(1-\delta)n}=\frac{1+\delta}{1-\delta}=\mbox{II}_\infty(\delta). \]
    \noindent{For Condition III, we first will rewrite the limit in terms of $c$. Since we know that}
        \[ \frac{(1-\delta)n}{2}\leq c<\frac{(1-\delta)n}{2}+1 \]
    \noindent{it follows that}
        \[ \frac{2(c-1)}{1-\delta}<n\leq\frac{2c}{1-\delta}. \]
    \noindent{Therefore, asymptotically we see}
        \[ n\sim\frac{2c}{1-\delta}. \]
    \noindent{Then, in the limit}
        \[ \frac{n-c}{c+1}=\frac{\frac{2c}{1-\delta}-c}{c+1}=\frac{c(1+\delta)}{(c+1)(1-\delta)} \]
    \noindent{so the limit becomes}
        \[ \lim_{n\to\infty}\left(\frac{\left(\frac{1+\delta}{1-\delta}\right)^{\frac{1}{2}}{{n}\choose{c}}}{1-\left(\frac{\varepsilon}{\lambda}\frac{n-c}{c+1}\right)^2}\right)^{\frac{1}{c}}= \lim_{c\to\infty}\left(\frac{\left(\frac{1+\delta}{1-\delta}\right)^{\frac{1}{2}}{{\left\lceil\frac{2c}{1-\delta}\right\rceil}\choose{c}}}{1-\left(\frac{\varepsilon}{\lambda}\frac{c(1+\delta)}{(c+1)(1-\delta)}\right)^2}\right)^{\frac{1}{c}}=  \lim_{c\to\infty}{{\left\lceil\frac{2c}{1-\delta}\right\rceil}\choose{c}}^{\frac{1}{c}}. \]
    \noindent{There are no problems with $\lambda$ being in the limit since it is simply a constant that was assumed to satisfy some conditions. Next, we use Stirling's approximation, which states that}
    \[n! \sim \sqrt{2\pi n}\left(\frac{n}{e}\right)^n. \]
    \noindent{Thus, by Stirling's approximation, the aforementioned limit becomes}
        \begin{align*}
            \lim_{c\to\infty}\left(\frac{\sqrt{2\pi\frac{2c}{1-\delta}}\left(\frac{2c}{(1-\delta)\mathrm{e}}\right)^{\frac{2c}{1-\delta}}}{\sqrt{2\pi\frac{c(1+\delta)}{1-\delta}}\left(\frac{c(1+\delta)}{\mathrm{e}(1-\delta)}\right)^{\frac{c(1+\delta)}{1-\delta)}}\sqrt{2\pi c}\left(\frac{c}{\mathrm{e}}\right)^c}\right)^{\frac{1}{c}} &= \lim_{c\to\infty}\frac{\left(\frac{2c}{(1-\delta)\mathrm{e}}\right)^{\frac{2}{1-\delta}}}{\left(\frac{c(1+\delta)}{\mathrm{e}(1-\delta)}\right)^{\frac{1+\delta}{1-\delta}}c^{\frac{1}{2c}}\left(\frac{c}{\mathrm{e}}\right)} \\
            &= \frac{\left(\frac{2}{(1-\delta)\mathrm{e}}\right)^{\frac{2}{1-\delta}}}{\left(\frac{1+\delta}{(1-\delta)\mathrm{e}}\right)^{\frac{1+\delta}{1-\delta}}} \mathrm{e} \\
            &= \mbox{III}_\infty(\delta)
        \end{align*}
    \noindent{Plotting $\mbox{I}_\infty(\delta)$, $\mbox{II}_\infty(\delta)$, and $\mbox{III}_\infty(\delta)$ on $0<\delta<1$ reveals that the lowest point which still lies above all three curves occurs at approximately $(0.280041,6.14182)$, the intersection of $\mbox{I}_\infty$ and $\mbox{III}_\infty$. Thus, the optimal value for $\delta$ in the infinite case is $\delta\approx 0.280041$. There is good reason to believe this $\delta$ is optimal for finite $n$ as well. \\}

    \noindent{\emph{Part 3: The finite case.} With the value of $\delta$ now specified, we return to our Conditions I, II, and III. With $\delta\approx 0.280041$, our Conditions become:}
        \[ \textrm{I}: \lambda>6.14182\varepsilon \textrm{\hspace{10mm}} \textrm{II}: \lambda>\varepsilon \left(        \frac{1.28004n}{2+0.719959} \right) \textrm{\hspace{10mm}} \textrm{III}: \lambda>\varepsilon\cdot\left(\frac{1.33339{{n}\choose{c}}}{1-\left(\frac{\varepsilon}{\lambda}\frac{n-c}{c+1}\right)^2}\right)^{\frac{1}{c}} \]
    \noindent{where $c=\lceil0.35998n\rceil$. However, by Condition I we know that at least $\varepsilon< 0.162818 \lambda$, and hence we can replace Condition III with a simplified variant:}
        \[ \lambda>\varepsilon\left(\frac{1.33339{{n}\choose{c}}}{1-0.0265098\left(\frac{n-c}{c+1}\right)^2}\right)^{\frac{1}{c}}. \]
    \noindent{All of these are explicit conditions on $\lambda$ in terms of $n$. An analysis reveals that I $>$ II and I $>$ III for all $n$, with III tending to I from below as $n\to\infty$ (by \emph{Part 2}). Hence, picking $\lambda>6.15\varepsilon$ gives the result of the theorem.}
    \end{proof}

\subsection{Proof of Theorem \ref{prop:RootDistribution}}

\noindent{We first introduce a useful lemma and prove it.}
    \begin{lemma}[Maximum Coefficient]
   \noindent{ \textit{Let $n>2$ be an integer, and let $\mathcal{M}>0$. Then, the function $$A(k)=\mathcal{M}^k{n \choose k},$$ where $1\leq k\leq n-1$, attains its maximum value at $k=m$, where $$m=\min \left\{ n-1, \max \left\{ \left\lceil \frac{\mathcal{M} n-1}{\mathcal{M}+1} \right\rceil, 1 \right\} \right\}.$$}\\}
    \noindent{\textit{Furthermore, $A(k)$ is increasing for $k < m$ and decreasing for $k > m.$}}
\end{lemma}
    \begin{proof}

    \noindent{Consider the function $S(k)=A(k+1)-A(k)$, which is analogous to the derivative of $A(k)$. By simple calculation,}
        \begin{align*}
             S(k)=A(k+1)-A(k)&=\frac{n!\mathcal{M}^{k+1}}{(k+1)!(n-k-1)!}-\frac{n!\mathcal{M}^k}{k!(n-k)!} \\
             &=\frac{n!\mathcal{M}}{(k+1)!(n-k)!}(\mathcal{M}n-1-k(\mathcal{M}+1)).
        \end{align*}
    \noindent{The funtion which determines the sign of $S(k)$ is a linear function in $k$, meaning that $A(k)$ has exactly one maximum over $1\leq k\leq n-1$ (or some $k$ such that both $k$ and $k+1$ are maxima). Moreover, the slope of that linear function is negative, so $A(k)$ is increasing for $k<m$ and decreasing for $k>m$. For $k=1$, if $S(k) \leq 0$, then $S(k)<0$ for all subsequent $k$, so the maximum should occur at $k=1$. $S(k) \leq 0$ implies that}
        \[ \frac{\mathcal{M}n-1}{\mathcal{M}+1} \leq 0. \]
    \noindent{Thus, $m=1$ according to the formula, as expected. If, on the other hand, $S(1)>0$, then we check the other boundary. If $S(n-2)> 0$, then we expect the maximum to occur at $k=n-1$. $S(n-2)> 0$ implies that}
        \[ \frac{\mathcal{M} n-1}{\mathcal{M}+1}>n-2 \]
    \noindent{which is precisely what we need to show. Thus,}
        \[ \left\lceil \frac{\mathcal{M} n-1}{\mathcal{M}+1} \right\rceil \geq n-1 \]
    \noindent{so the formula gives $m=n-1$, as expected. If, instead, $S(n-2)=0$, then we expect both $k=n-2$ and $k=n-1$ to be maxima. $S(n-2)=0$ implies that}
        \[ \frac{\mathcal{M} n-1}{\mathcal{M}+1}=n-2 \]
    \noindent{and hence the formula gives $m=n-2$, as expected. Now that the boundary cases have been checked, assume that $S(1)>0$ and that $S(n-2)<0$. This means that the maximum of $A(k)$ occurs for some $2 \leq k \leq n-2$. We know that the maximum would occur for the first value of k such that $S(k)<0$ (call that value $m$), since in that case for all previous $k<m$'s, either $S(k)>0$, in which case that $k$ is not the maximum, or $S(k)=0$, in which case $k=j-1$, and both $k$ and $j$ are the maxima. We know that the sign component of $S(k)$, $s(k)=\mathcal{M}n-1-k(\mathcal{M}+1)$, when considered as a linear function $s(x)$, changes its sign at $s(x)=0$, or, equivalently, at} 
        \[ x=\frac{\mathcal{M}n-1}{\mathcal{M}+1}. \]
    \noindent{Hence, the first $k$ for which $S(k)<0$ is precisely}
        \[ k = \left\lceil \frac{\mathcal{M} n-1}{\mathcal{M}+1} \right\rceil \]
    \noindent{as desired.}
    
    \end{proof}
    
	\noindent{Now we prove Theorem 2.2, the exponential convergence of the polynomials in Dirichlet space, when their roots have $n$-th power mean  sufficiently small.}
	
	\begin{proof} Applying Lemma 5.2, we wish to show that
        \[ \frac{1}{2} \| F_{\lambda}(z) \|_{\mathcal{D}}^2 =  \frac{1}{2} \sum_{k=1}^{n} |c_k|^2 k \geq \| G(z) \|_{\mathcal{D}}^2 =\sum_{k=1}^{n-1} |c_{n-k}|^2k = \sum_{k=1}^{n} |c_k|^2 (n-k)\]
        \noindent{when} \[\left(\sqrt[n]{\frac{1}{n}\sum_{i=1}^n |\alpha_i/\lambda|^n|}\right) \le \frac{4}{27}.\]
    \noindent{Suppose $n > 2$ and $\max |\alpha_i| < \varepsilon$. Combining both sides of the inequality gives us}
        \[ \sum_{k=1}^{n} |c_k|^2 \left( \frac{3}{2} k - n \right) \geq 0 \]
    \noindent{which can be split into}
        \[ \sum_{k=\lfloor\frac{2}{3}n\rfloor +1}^{n} |c_k|^2 \left( \frac{3}{2} k - n \right) \geq \sum_{k=1}^{\lfloor\frac{2}{3} n\rfloor} |c_k|^2 \left( n - \frac{3}{2} k \right) \]
    \noindent{By Lemma 5.1, we can bound the right-hand side above by}
        \[ \sum_{k=1}^{\lfloor\frac{2}{3} n\rfloor} |c_k|^2 \left( n - \frac{3}{2} k \right) \leq \sum_{k=1}^{\lfloor\frac{2}{3}n\rfloor} \left[{n \choose k} \left(\frac{1}{n}\sum_{i=1}^n |\alpha_i/\lambda|^n|\right)^{1-k/n}\right]^2 \cdot  \left( n - \frac{3}{2} k \right). \]
       \noindent{Note that Lemma 5.3 states that} 
            \[ \binom{n}{k}\left(\sqrt[n]{\frac{1}{n}\sum_{i=1}^n |\alpha_i/\lambda|^n|}\right)^{-k} \]
       \noindent{is maximized at $k=m,$ where}
            \[ \mathcal{M} = \left(\sqrt[n]{\frac{1}{n}\sum_{i=1}^n |\alpha_i/\lambda|^n|}\right)^{-1} \ge \frac{27}{4}, \]
        \noindent{and}
       \begin{align*}
             m&=\min \left\{ n-1, \max \left\{ \left\lceil \frac{\mathcal{M} n-1}{\mathcal{M}+1} \right\rceil, 1 \right\} \right\} \\
             &\le \min \left\{ n-1, \max \left\{ \left\lceil \frac{27n-4}{31} \right\rceil, 1 \right\} \right\} \le \frac{2}{3}n.
        \end{align*}
       
      \noindent{Lemma 5.3 says more than that; it says that the function is increasing for $k < m.$ In particular, the function} 
        \[ \binom{n}{k}\left(\sqrt[n]{\frac{1}{n}\sum_{i=1}^n |\alpha_i/\lambda|^n|}\right)^{n-k} \]
        \noindent{from $k = 1$ to $k = \lfloor\frac{2}{3}n\rfloor$ is maximized at $k = \lfloor\frac{2}{3}n\rfloor.$ Thus, we can bound,}
        \begin{align*}
            \sum_{k=1}^{\lfloor\frac{2}{3} n\rfloor} |c_k|^2 \left( n - \frac{3}{2} k \right) &\leq \sum_{k=1}^{\lfloor\frac{2}{3}n\rfloor} \binom{n}{\frac{2n}{3}}^2\left(\frac{1}{n}\sum_{i=1}^n |\alpha_i/\lambda|^n|\right)^{\frac{2}{3}} \cdot  \left( n - \frac{3}{2} k \right) \\
            &\leq{n \choose \frac{2n}{3}}^2 \left( \frac{1}{n}\sum_{i=1}^n |\alpha_i/\lambda|^n|\right)^{\frac{2}{3}} \cdot \frac{n^2}{2}.
        \end{align*}
       \noindent{According to Stirling, we have the following bound}
       \[\sqrt{2\pi}n^{n+\frac{1}{2}}e^{-n}\le n! \le en^{n+\frac{1}{2}}e^{-n}\]
       \noindent{which bounds $\binom{n}{\frac{2n}{3}}$ above}
       \begin{align*}\frac{n!}{(2n/3)!(n/3)!} &\le \frac{en^{n+\frac{1}{2}}e^{-n}}{\left(\sqrt{2\pi}(\frac{2n}{3})^{\frac{2n}{3}+\frac{1}{2}}e^{-\frac{2n}{3}}\right)\left(\sqrt{2\pi}(\frac{n}{3})^{\frac{n}{3}+\frac{1}{2}}e^{-\frac{n}{3}}\right)} \\
       &\le  \frac{e}{2\pi}\left(\frac{1}{2}\right)^{\frac{2}{3}n+\frac{1}{2}}3^{n+1}\cdot\frac{1}{\sqrt{n}} .\end{align*}
       \noindent{Our above expression is therefore bounded above by}
        \[ \left( \frac{e}{2\pi}\left(\frac{1}{2}\right)^{\frac{2}{3}n+\frac{1}{2}}3^{n+1}\cdot\frac{1}{\sqrt{n}} \right)^2 \left( \frac{1}{n}\sum_{i=1}^n |\alpha_i/\lambda|^n| \right)^{\frac{2}{3}} \cdot \frac{n^2}{2}. \]
    \noindent{Similarly, we can bound the left-hand side below by}
        \[ \frac{n}{2} \leq \sum_{k=\lfloor\frac{2}{3}n\rfloor +1}^{n} |c_k|^2 \left( \frac{3}{2} k - n \right). \]
    \noindent{Combining these two, it suffices to show}
        \[ \frac{n}{2} \geq \left( \frac{e}{2\pi}\left(\frac{1}{2}\right)^{\frac{2}{3}n+\frac{1}{2}}3^{n+1}\cdot\frac{1}{\sqrt{n}} \right)^2 \left( \frac{1}{n}\sum_{i=1}^n |\alpha_i/\lambda|^n| \right)^{\frac{2}{3}} \cdot \frac{n^2}{2}\]
    \noindent{which holds true for}
    \[\sqrt[n]{\frac{1}{n}\sum_{i=1}^n |\alpha_i/\lambda|^n} \le \frac{4}{27}\left(\frac{2\sqrt{2}\pi}{3e}\right)^{3/n}.\]
    
    \noindent{In particular, this is satisfied when} 
    
    \[\lambda \geq 6.75\sqrt[n]{\frac{1}{n}\sum_{i=1}^n |\alpha_i|^n}\]
    
    \noindent{proving Theorem 2.2. One can also slightly generalize Theorem 2.1 in another way: by appending a large root. Consider $F_\lambda(z)(z-M)=B(z)\cdot G(z),$ where $|M|>1,\delta<1$ and}
    \[|M|\delta \ge \binom{n}{n/2}\cdot \left(\frac{\varepsilon}{\lambda}\right)^{n/2} \ge |c_k|\]
    
    \noindent{for all $0 \le k \le n-1.$ Observe that the coefficient of $z^k$ in $(z-M)F_\lambda(z)$ is now   $c_{k-1}-Mc_k$, which satisfies}
    \[|c_{k-1}-Mc_k| \le |Mc_k| + |M\delta| \le |M||c_k+\delta|.\]
    
    \noindent{Furthermore, the coefficient of $z^n$ is now $c_{n-1}-M$, which satisfies}
    \[|c_{n-1}-M| \ge |M|-|c_{n-1}| \ge |M|-|M\delta| = |M(1-\delta)|.\]

    \noindent{Through similar algebra as in the proof for the first part of the statement, one finds that the desired inequality on the Dirichlet norms holds true for}
    \[\sqrt[n]{\frac{1}{n}\sum_{i=1}^n |\alpha_i/\lambda|^n} \le \frac{4}{27}\left[\frac{2\sqrt{2}\pi}{3e} \cdot \left(1-\delta \sqrt{n}-\delta \right)\right]^{3/n}.\]

	\end{proof}

\subsection{Proof of Proposition \ref{thm: Vlad}}

    \noindent{We first prove the following Lemma.}
    
    \begin{lemma}
       \noindent{Let $F$ be a monic polynomial of degree $n$. Then $||F||^2_\mathcal{D} \geq n$.}
    \end{lemma}
    \begin{proof}
        \noindent{By definition,}
            \[ \| f \|_{\mathcal{D}}^2 = \sum_{k=1}^{n} k |c_k|^2 = n + \sum_{k=1}^{n-1} k |c_k|^2. \]
        \noindent{Since $\forall 1\leq k \leq n-1$, $k|c_k|^2 \geq 0$, the result follows.}
    \end{proof}
    \noindent{Now we prove the Proposition.}
    \begin{proof}
        \noindent{Consider the decompositions $F(z) = B_{\lambda}(z) \cdot G(z)$ and $F_{\lambda^2}(z) = B(z) \cdot G_{\lambda^2}(z)$. For sufficiently large $\lambda \in \R$, the function $F_{\lambda^2}(z)$ approaches $z^n$, so by continuity we know that}
            \[ \left\| F_{\lambda^2}(z) \right\|_{\mathcal{D}}^2 \longrightarrow \left\| z^n \right\|_{\mathcal{D}}^2 = n. \]
        \noindent{Also, by Proposition \ref{Blaschke-Blaschke Equivalency}, we know that}
            \[ \frac{F_{\lambda^2}(z)}{B(z)} = \frac{1}{\lambda^n} \cdot \frac{F(z)}{B_{\lambda}(z)}. \]
        \noindent{Hence, we have}
            \[ \left\| \frac{1}{\lambda^n} \cdot \frac{F(z)}{B_{\lambda}(z)} \right\|_{\mathcal{D}}^2 \leq \frac{n}{2}. \]
        \noindent{However, by Lemma 5.1, we know that}
            \[ n \leq \left\| F(z) \right\|_{\mathcal{D}}^2.  \]
        \noindent{Therefore, we find that}
            \[ \left\| \frac{1}{\lambda^n} \cdot \frac{F(z)}{B_{\lambda}(z)} \right\|_{\mathcal{D}}^2 \leq \frac{1}{2} \left\| F(z) \right\|_{\mathcal{D}}^2. \]
    \end{proof}

\subsection{Proof of Theorem \ref{ScalingCorollary}}
	
    \begin{proof}
        \noindent{Suppose that $\max \{|\alpha_i|\} < \gamma < \lambda$ and note that then}
            \[ F_{\gamma}(\gamma z/\lambda) = \prod_{i=1}^{n} \left( \frac{\gamma}{\lambda} z - \frac{\alpha_i}{\gamma} \right) = \left( \frac{\gamma}{\lambda} \right)^n \prod_{i=1}^{n} \left( z - \frac{\lambda}{\gamma^2} \alpha_i \right).\] 
        \noindent{Therefore,}    
            \[ G_{\gamma}(z)=\left( \frac{\gamma}{\lambda} \right)^n \prod_{i=1}^{n} \frac{1}{\lambda} \left( \lambda^2 - \frac{\lambda}{\gamma^2} \overline{\alpha_i}z \right) \]
        \noindent{since $|\alpha_i| < \lambda, \gamma$ for all $i \in I$. Via direct computation we find that}
            \[ G_{\lambda} = \prod_{i=1}^{n} \frac{1}{\gamma} \left( \gamma^2 - \frac{\overline{\alpha_i}}{\lambda} z \right) = \prod_{i=1}^{n} \gamma \left( 1 - \frac{\overline{\alpha_i}}{\lambda \gamma^2} z \right) = \left( \frac{\gamma}{\lambda} \right)^n \prod_{i=1}^{n} \frac{1}{\lambda} \left( \lambda^2 - \frac{\lambda}{\gamma^2} \overline{\alpha_i}z \right) = G_{\gamma} (z).\]
        \noindent{Therefore, if all roots are captured during $\gamma$-Blaschke and $\lambda$-Blaschke factorization,}
            \[ G_{\gamma}(z) = G_{\lambda} (z).\]
        \noindent{In the second scenario, let $m$ be such that $|\alpha_i| < \gamma \lambda$ for $i=1$ to $i=m$, and $|\alpha_i| > \gamma \lambda$ for $i=m+1$ to $i=n$. We deduce that}
           \begin{align*} 
            \gamma^n G_{\gamma}(\lambda z) &= \gamma^n \prod_{i=1}^{m} \frac{1}{\lambda} \left( \lambda^2 - \frac{\overline{\alpha_i}}{\gamma} \left( \lambda z \right) \right) \prod_{i=m+1}^n \left(\lambda z - \frac{\alpha_i}{ \gamma}\right)\\&=\gamma^n \prod_{i=1}^{m} \left( \lambda - \frac{\overline{\alpha_i}}{\gamma} z \right) \prod_{i=m+1}^n \left(\lambda z - \frac{\alpha_i}{ \gamma}\right) \\ &= \gamma^n \prod_{i=1}^{m} \frac{\lambda}{\gamma} \left( \gamma - \frac{\overline{\alpha_i}}{\lambda} z \right)\prod_{i=m+1}^n \frac{\lambda}{\gamma}\left(\gamma z - \frac{\alpha_i}\lambda\right)  \\ &= \lambda^n \prod_{i=1}^{m} \left( \gamma - \frac{\overline{\alpha_i}}{\lambda} z \right)\prod_{i=m+1}^n \left(\gamma z -\frac{\alpha_i}{\lambda}\right) = \lambda^n G_{\lambda} (\gamma z).
        \end{align*}
        \noindent{Therefore, we have the alternate relation by which, without any assumptions on the norms of the roots,}
            \[\gamma^n G_{\gamma}(\lambda z) = \lambda^n G_{\lambda} (\gamma z).  \]
    \end{proof}

\subsection{Proof of Proposition \ref{thm: ShrunkRoots}}
    
    \begin{proof}
        \noindent{Consider Theorem \ref{ScalingCorollary} and let $\gamma = 1$. Then, the equations hold.}
    
    \end{proof}
    
\subsection{Proof of Proposition \ref{def: LogDeriv}}
        \begin{proof}
            \noindent{Firstly, recall that}
                \begin{align*}
                    B_r(z) &= \prod_{i=1}^{n} \frac{(z-\alpha_i)r}{(r^2 - \overline{\alpha}z)}.
                \end{align*}
            \noindent{Taking the derivative and applying the Chain Rule, we find that}
            \begin{align*}
                B_r'(z) &= \sum_{j=1}^{n} \frac{r(r^2-|\alpha_j|^2)}{(r^2 - \overline{\alpha_j}z)^2} \prod_{\substack{i=1 \\ i \neq j}}^{n} \frac{(z-\alpha_i)r}{(r^2 - \overline{\alpha_i}z)}.
            \end{align*}
        \noindent{Dividing through by an $r$-Blaschke product, we obtain the logarithmic derivative}
            \begin{align*}
                \frac{B_r'(z)}{B_r(z)} &= \sum_{j=1}^{n} \frac{r(r^2-|\alpha_j|^2)}{(r^2 - \overline{\alpha_j}z)^2} \frac{(r^2 - \overline{\alpha_j}z)}{(z-\alpha_j)r} = \sum_{j=1}^{n} \frac{(r^2-|\alpha_j|^2)}{(r^2 - \overline{\alpha_j}z)(z-\alpha_j)}.
            \end{align*}
            \noindent{Let $z = re^{it}$. Then,}
            \begin{align*}
                \frac{B_r'(re^{it})}{re^{-it} B_r(re^{it})} &= \sum_{j=1}^{n} \frac{(r^2-|\alpha_j|^2)}{\overline{z}(r^2 - \overline{\alpha_j}z)(z-\alpha_j)} \\
                &= \sum_{j=1}^{n} \frac{(r^2-|\alpha_j|^2)}{r^2(\overline{z} - \overline{\alpha_j})(z-\alpha_j)} \\
                &= \sum_{j=1}^{n} \frac{(r^2-|\alpha_j|^2)}{r^2 |z - \alpha_j|^2}.
            \end{align*}
            \noindent{Therefore,}
            \begin{align*}
                |B_r'(re^{it})| &= \sum_{j=1}^{n} \frac{(r^2-|\alpha_j|^2)}{r |z - \alpha_j|^2}.
            \end{align*}
        \end{proof}

\subsection{Proof of Theorem \ref{OneStepR}}

    \begin{proof}
         Recall that, if $|z| = r$, then $|z-a|^2 = r^{-2}|r^2 - \overline{a} z |^2$. Now, let
            \[ f = (z-\alpha)F \] 
        \noindent{and}
            \[ g = \frac{1}{r} \left( r^2 - \overline{\alpha} z \right) F. \]
        \noindent{Then,}
            \[ f' = F + (z-\alpha) F' \]
        \noindent{and}
            \[ g' = - \frac{\overline{\alpha}}{r} F + \frac{1}{r} \left( r^2 - \overline{\alpha}z \right) F'. \]
        \noindent{Via direct computation, we find that}
            \begin{align*}
                |f'|^2 &= |F|^2 + F \overline{(z-a)F'} + \overline{F} (z-a)F' + |z-a|^2 |F'|^2
            \end{align*}
        \noindent{and}
            \begin{align*}
                |g'|^2 &= \frac{|a|^2}{r^2} |F|^2 - \frac{a}{r^2} \overline{F} \left( r^2 - \overline{a} z \right) F' - \frac{\overline{a}}{r^2} F \overline{\left( r^2 - \overline{a}z \right) F'} + \frac{1}{r^2} \left| r^2 - \overline{a}z \right| |F'|^2.
            \end{align*}
        \noindent{Thus, integrating over the boundary of a disk of radius $r$, we find that}
            \begin{align*}
                \int_{\partial \D_r} |f'|^2 - \int_{\partial \D_r} |g'|^2 &= \left( 1 - \frac{|a|^2}{r^2} \right) \int_{\partial \D_r} |F|^2 + \int_{\partial \D_r} \left[ F \overline{(z-a)F'} + \overline{F} (z-a)F' \right] \\
                & \hspace{37mm} + \left. \frac{a}{r^2} \overline{F} \left( r^2 - \overline{a}z \right) F' + \frac{\overline{a}}{r^2} F \overline{\left( r^2 - \overline{a}z \right)F'} \right] \\
                &= \left( 1 - \frac{|a|^2}{r^2} \right) \int_{\partial \D_r} |F|^2 + \left( 1 - \frac{|a|^2}{r^2} \right) \int_{\partial \D_r} \left[ F\overline{F'}\overline{z} + \overline{F} F' z \right] \\
                & \geq \left( 1 - \frac{|a|^2}{r^2} \right) \int_{\partial \D_r} |F|^2.
            \end{align*}
        \noindent{Therefore,}
            \[ \int_{0}^{2\pi} \left| g'(re^{it}) \right|^2 \: dt \leq \int_{0}^{2\pi} \left| f'(re^{it}) \right|^2 \: dt - \left( 1 - \frac{|a|^2}{r^2} \right) \int_{0}^{2\pi} \left| F(re^{it}) \right|^2 \: dt. \] \\
        \end{proof}
        
        \section*{Acknowledgements}

        \noindent{The authors were supported by the 2018 Summer Math Research at Yale (SUMRY) program. The authors would like to thank Stefan Steinerberger and Hau-Tieng Wu for their endless mentorship and Lihui Tan for helpful feedback on a preliminary draft of this paper.}

\bibliography{Blaschke}
\bibliographystyle{alpha}

\end{document}